\documentclass[ACS,STIX1COL,proof]{WileyNJD-v2}

\usepackage{amsfonts,amssymb,stmaryrd,amsmath,amsthm,amssymb}
\usepackage{graphicx,color}
\usepackage{multimedia}
\usepackage{float}
\usepackage{placeins}
\usepackage{todo}
\usepackage{algorithm}

\usepackage{float}

\articletype{Article Type}%

\received{30 August 2016}
\revised{6 July 2017}
\accepted{17 October 2017}

\newcommand \p {\partial}

\newcommand \R {\mathbb{R}}

\renewcommand \L {\mathrm{L}}
\newcommand \W {\mathrm{W}}

\renewcommand \H {\mathrm{H}}

\newcommand \I {\mathrm{I}}

\renewcommand \d {\mathrm{d}}

\DeclareMathOperator{\divg}{div}

\begingroup\makeatletter\ifx\SetFigFont\undefined%
\gdef\SetFigFont#1#2#3#4#5{%
	\reset@font\fontsize{#1}{#2pt}%
	\fontfamily{#3}\fontseries{#4}\fontshape{#5}%
	\selectfont}%
\fi\endgroup%

\begin{document}

\title{Optimal Control for a Class of Infinite Dimensional Systems Involving an $L^\infty$-term in the Cost Functional}

\author[1]{S\'ebastien Court*}

\author[2]{Karl Kunisch}

\author[1]{Laurent Pfeiffer}

\authormark{AUTHOR ONE \textsc{et al}}

\address[1]{Institute for Mathematics and Scientific Computing, Karl-Franzens-Universit\"{a}t, Heinrichstr. 36, 8010 Graz, Austria, email: {\tt sebastien.court@uni-graz.at}, {\tt laurent.pfeiffer@uni-graz.at}.}

\address[2]{Institute for Mathematics and Scientific Computing, Karl-Franzens-Universit\"{a}t, Heinrichstr. 36, 8010 Graz, Austria, and Radon Institute, Austrian Academy of Sciences, email: {\tt karl.kunisch@uni-graz.at}.}

\corres{*Corresponding author \email{sebastien.court@uni-graz.at}}

%\presentaddress{This is sample for present address text this is sample for present address text}

\abstract[Summary]{An optimal control problem with a time-parameter is considered. The functional to be optimized includes the maximum over time-horizon reached by a function of the state variable, and so an $L^\infty$-term. In addition to the classical control function, the time at which this maximum is reached is considered as a free parameter. The problem couples the behavior of the state and the control, with this time-parameter. A change of variable is introduced to derive first and second-order optimality conditions. This allows the implementation of a Newton method. Numerical simulations are developed, for selected ordinary differential equations and a partial differential equation, which illustrate the influence of the additional parameter and the original motivation.}


\maketitle


%\footnotetext{\textbf{Abbreviations:} ANA, anti-nuclear antibodies; APC, antigen-presenting cells; IRF, interferon regulatory factor}

\section{Introduction}

We consider optimal control problems with a cost functional involving a function of the state variable at some time $\tau$ during its time evolution. The time $\tau$ itself is free to move within the time-span $(0,T)$ of the whole experiment. As a very first thought, one might think of the prototype of a new vehicle and ask what is the highest reachable speed on the race track.
Our original motivation for the class of problems under consideration stems from the domain of cardiac electro-physiology. The problem of fibrillation of a part of the heart can be mechanically tackled with the use of electric shocks acting on this muscle, leading to the re-oxygenation of the ill area, and by this means forcing this area to recover a healthy electric activity.
The efficiency of the defibrillation is known to be related to the maximum reached over time by the derivative of a pressure in the heart. This quantity can be mathematically formulated as the maximum taken by a function of the state variable of the problem.
\textcolor{black}{We refer to \cite{CKP16} for more details on related optimal control problems arising in electro-cardiology.}
While we continue to work on this challenging problem, we focus in this article on the methodology enabling to derive optimality conditions for such a problem, for a simpler class of partial differential equations.

The optimal control problem that we shall investigate can be formulated as follows:
\begin{eqnarray} \label{eqPbInit}
\left\{\begin{array} {l}
\displaystyle \max_{\tau \in [0,T],\, u \in \L^2(0,T;U)} \int_0^T \ell(y(t),u(t)) \d t + \phi_1(y(\tau)) + \phi_2(y(T)),
\\
\qquad \text{subject to: } \quad
\dot{y}= f(y,u), \quad y(0)=y_0, \quad
\displaystyle \textcolor{black}{\int_0^T \| u(t) \|_U^2 \d t \leq \gamma},
\end{array} \right.
\end{eqnarray}
where the state variable $y$ is the solution of an evolution equation controlled by $u$.
In this problem, $\ell$ denotes the cost functional, $\phi_1$ the functional we want to maximize at some time $\tau$, and $\phi_2$ is the terminal cost. The analytical framework will be specified later. The specificity of this kind of problem lies in the fact that a time parameter, namely $\tau$, can be optimized. Not only do we maximize $\phi_1\circ y$ with the help of the control $u$, but we also optimize the time $\tau$ for which the maximum is reached.
Note that, when $\phi_1$ is nonnegative, the problem can be equivalently formulated in the following way:
\begin{eqnarray*}
 \left\{\begin{array} {l}
  \displaystyle
  \max_{u \in \L^2(0,T;U)} \int_0^T \ell(y(t),u(t)) \d t + \| \phi_1(y(\cdot)) \|_{\L^{\infty}(0,T;\R)} + \phi_2(y(T)), \\
  \text{subject to: } \quad \dot{y}= f(y,u), \quad y(0)= y_0,
  \quad
  \displaystyle \textcolor{black}{\int_0^T \| u(t) \|_U^2 \d t \leq \gamma}.
 \end{array} \right.
\end{eqnarray*}
In this fashion, the cost function incorporates the $\L^\infty$-norm of a given function of the state variable.

\textcolor{black}{
 With appropriate technical modifications, the results provided in this paper can be extended to the following hybrid control problem:}
\begin{eqnarray}
\left\{\begin{array} {l}
\displaystyle   \max_{
 \begin{subarray}{c} 0= \tau_0 < \tau_1 < ... < \tau_K= T \\ %\notag
 u \in \L^2(0,T;U)
 \end{subarray}
} \ \sum_{k=1}^K \ \int_{\tau_{k-1}}^{\tau_k} \ell_k(y(t),u(t)) \d t + \phi_k(y(\tau_k)), \\
\text{subject to: }  \quad
\dot{y}(t)= f_k(y(t),u(t)), \text{ for a.\,e.\,} t \in (\tau_{k-1},\tau_k), \quad
y(0)= y_0, \quad \textcolor{black}{\displaystyle \int_0^T \| u(t) \|_U^2 \d t \leq \gamma.}
\end{array} \right. \label{eqHybrid}
\end{eqnarray}
The above problem incorporates time parameters (or switching times) $\tau_1$,...,$\tau_{K-1}$ which can be optimized. The term {\it hybrid} refers here to the fact that at the switching times $\tau_k$, the system move from a given regime (described here by the dynamics $f_k$) to another one (described by the dynamics $f_{k+1}$). At the switching times, the integral cost changes and a function of the state is also incorporated into the cost function. There are many ways to generalize problem~\eqref{eqHybrid} (see e.g.~\cite{GP05}). For example, one can consider a formulation of the problem for which the dynamics $f_k$ (used during the interval $(\tau_{k-1},\tau_k)$) can be itself chosen into a finite set of functions. Such generalizations are beyond the scope of the paper.

In the optimal control literature, many problems include time parameters that have to be optimized. Among them, time-optimal control problems have probably attracted the most the attention. These problems basically consist in minimizing the time needed to reach a given target. We refer to the early reference~\cite{HL69} on this topic. Time-optimal control problems have been studied for various models: See for example~\cite{KW13} for the wave equation,~\cite{KR15,KPR16} for the monodomain system,~\cite{Bar97} for the Navier-Stokes equations,~\cite{MRT12} for the heat equation. We also mention the time-crisis management problem studied in~\cite{BR16}; For such a problem, the time spent out of a certain closed domain must be minimized.

For our problem, the first-order optimality conditions consist of a weak maximum principle for the control variable, and a transversality condition for the optimality of the time parameter.
The derivation of the transversality condition is difficult, in so far as the Lagrangian of the problem is not differentiable with respect to the time-parameters. In e.g.~\cite{IK10,RZ99,RZ00}, a change of variables (in time) is performed to derive the transversality condition in the case of a time-optimal control problem. This is the approach that we adopt here. It also enables us to derive second-order optimality conditions, as in \cite{KPR16} and \cite{LiGaoWang} for finite-dimensional hybrid control problems. In this last reference, a numerical test based on a Riccati equation is provided to check the sufficient second-order conditions. In~\cite{GP05}, specific needle variations are designed, for hybrid systems. Other approaches can also be considered. In~\cite{BR16}, the time-crisis management problem is regularized and optimality conditions are derived by $\Gamma$-convergence.
\textcolor{black}{One of the contributions of the present paper is that the concept of reparameterization for the class of hybrid systems speficifed in \eqref{eqPbInit} are systematically investigated in the infinite-dimensional setting, thus allowing for applications in the context of PDEs.}

The paper is organized as follows: In section~\ref{sec2} the functional framework is specified, and the problem is transformed with a change of variables. Section~\ref{sec3} is devoted to the derivation of first and second-order optimality conditions. The abstract framework proposed here is shown to be satisfied, as an example, by the Navier-Stokes equations in dimension 2. The issue of numerical resolution is addressed in section~\ref{sec4}: We use the theoretical expressions of the optimality condition in order to design an algorithm which solves the problem, mixing Barzilai-Borwein gradient steps with Newton steps. As illustrations, we consider two examples of ordinary differential systems and one example dealing with the Burgers equation.

\paragraph{Notation.}
When there is no \textcolor{black}{ambiguity}, the time variable is sometimes omitted, or written only once (e.g. $H(y,u,p)(s)$, instead of $H(y(s),u(s),p(s))$. First and second-order partial derivatives are denoted with the use of indexes. The first and second-order derivatives with respect to all variables \textcolor{black}{(except the adjoint state $p$ and the Lagrange multiplier $\lambda$, for the Hamiltonian and the Lagrangians)} are denoted by $D$ and $D^2$, respectively. \textcolor{black}{The partial derivative of a function $\varphi$ with respect to a variable $y$ in a direction $z$ is denoted as $\varphi_y(y).z$}.
When a function $h$ is left- and right-continuous at a given time $t$, the left- and right-limits are denoted by $h(t^-)$ and $h(t^+)$, respectively, and the jump is denoted by $[h]_t$.

\section{Formulation of the problem} \label{sec2}

\subsection{Setting} \label{secset}

Let $X$ be a real Hilbert space, and $Y$ be a reflexive Banach space forming with $X$ a Gelfand triple $Y \subseteq X \equiv X' \subseteq Y'$, with $Y$ densely contained in $X$. We denote $Z = Y'$, so that $Z'$ is isomorphic to $Y$. Let $U$ be the Hilbert space of controls. Further let $\ell:X\times U \rightarrow \R$ and $f:Y\times U \rightarrow Y'$ be two twice continuously Fr{\'e}chet-differentiable mappings. For $\tilde{u} \in \L^2(0,T;U)$, the state is governed by the autonomous control system
\begin{eqnarray} \label{mainsys0}
\left\{ \begin{array} {l}
\dot{\tilde{y}} = f(\tilde{y},\tilde{u}) \quad \text{on } (0,T),\\
\tilde{y}(0) = y_0,
\end{array} \right.
\end{eqnarray}
where the initial condition $y_0 \in X$ is given. We assume that for all $y_0 \in X$ and $\tilde{u}\in \L^2(0,T;U)$ this system has a unique solution $\tilde{y}$ in the space:
\begin{eqnarray*}
 W(0,T;Y) := \L^2(0,T;Y) \cap \W^{1,2}(0,T;Y').
\end{eqnarray*}
Recall the continuous embedding $W(0,T;Y) \hookrightarrow \mathcal{C}([0,T];X)$, so that in particular the initial condition makes sense in $X$. \textcolor{black}{Weak solutions for systems of type~\eqref{mainsys0} will be understood in $W(0,T;Y)$ in the variational sense, that is to say for all $\varphi \in \L^2(0,T;Y)$ they must satisfy
 \begin{eqnarray*}
  \left\langle \dot{\tilde{y}}(t) - f(\tilde{y}(t),\tilde{u}(t)) ; \varphi(t) \right\rangle_{Y';Y}  =  0
 \end{eqnarray*}
 for almost every $t\in (0,T)$, and $\tilde{y}(0) = y_0$ in $X$.} Let $\phi_1$ and $\phi_2: X \rightarrow \R$ be two twice continuously Fr\'echet-differentiable mappings.
We consider the following problem:
\begin{eqnarray}
& & \left\{ \begin{array} {l}
\displaystyle \max_{\tau\in[0,T], \ \tilde{u} \in \L^2(0,T;U)} \ \int_0^T \ell(\tilde{y}(t),\tilde{u}(t)) \d t + \phi_1(\tilde{y}(\tau)) + \phi_2(\tilde{y}(T)) \\[5pt]
\displaystyle \text{subject to: }\quad
\dot{\tilde{y}} = f(\tilde{y},\tilde{u}) \text{ in } Y',\quad \tilde{y}(0) = y_0 \text{ in } X,
\quad \displaystyle \text{\textcolor{black}{$\int^T_0 \|\tilde{u}(t)\|_U^2 \, dt \leq \gamma$}},
\end{array} \right. \label{pbmain}
\end{eqnarray}
\textcolor{black}{where $\gamma\in (0,\infty)$. The operators $f$ and $\ell$ are assumed to generate Nemytskii operators on $W(0,T;Y) \times \L^2(0,T;U)$, whose classical properties can be found in~\cite{Troltzsch}, section~4.3, for instance. For a sake of compact notation, we keep the same notation for these Nemytskii operators so defined, and will consider in particular the relations $f(\tilde{y},\tilde{u})(t) = f(\tilde{y}(t),\tilde{u}(t))$ and $\ell(\tilde{y},\tilde{u})(t) = \ell(\tilde{y}(t),\tilde{u}(t))$.} Throughout the paper, we consider the following set of assumptions on $f$ and $\ell$.
\begin{description}
 \item[$\mathbf{A0}$] The mapping $\ell :  W(0,T;Y) \times \L^2(0,T;U) \rightarrow  \L^1(0,T;\R)$ is twice continuously Fr\'echet-differentiable.
 \item[$\mathbf{A1}$] The mapping $f : W(0,T;Y) \times \L^2(0,T;U) \rightarrow \L^2(0,T;Y')$
 is twice continuously Fr\'echet-differentiable.
 \item[$\mathbf{A2}$] For all $\tilde{u} \in \L^2(0,T;U)$, $y_0 \in X$, system~\eqref{mainsys0} admits a unique weak solution in $W(0,T;Y)$.
 \item[$\mathbf{A3}$]For all $(\tilde{y},\tilde{u}) \in W(0,T;Y) \times \L^2(0,T;U)$, $\xi \in \L^2(0,T;Y')$, $z_0\in X$, there exists a unique $z \in W(0,T;Y) $ solution to the following system
 \begin{eqnarray*}
  \left\{ \begin{array} {l}
   \dot{z} = f_{\tilde{y}}(\tilde{y},\tilde{u}).z + \xi \quad \text{on } (0,T), \\
   z(0) = z_0.
  \end{array} \right.
 \end{eqnarray*}
\end{description}

\textcolor{black}{We included the constraint $\|\tilde{u}\|_{\L^2(0,T;U)}^2 \leq \gamma$ on the energy of the control in the problem formulation~\eqref{pbmain}. While existence of solutions is not the focus of this paper, this constraint allows to argue existence in concrete cases, as illustrated in section~\ref{subsecexample}. In a follow-up work it can be of interest to include different constraints on the controls and/or the state variable, and to analyze their effect on the switching time $\tau$.}

\subsection{Comments on the functional framework}

\subsubsection{Example} \label{subsecexample}

The variational framework $W(0,T;Y)$ is applicable, for instance, to the Burgers system (see section~\ref{secnumPDE}), or to the unsteady Navier-Stokes system (see~\cite{HinzeKK01, HinzeKK}). Assumptions $\mathbf{A1}$--$\mathbf{A3}$ can be easily verified for the Burgers system, if we refer to the analysis of e.g.~\cite{Volkwein}.\\
Let us detail the verification for the Navier-Stokes system and also  formulate the problem in such a manner that existence of a solution can be guaranteed.
\begin{eqnarray*}
 %\begin{array}
 \dot{y} + (y\cdot \nabla) y - \nu \Delta y + \nabla p  =  Bu & \quad & \text{in } \Omega \times (0,T), \\
 \divg y  =  0 & \quad & \text{in } \Omega \times (0,T), \\
 y  =  0 & \quad & \text{on } \p \Omega\times (0,T), \\
 y(0)  =  y_0 & \quad & \text{in } \Omega .
 %\end{array}
\end{eqnarray*}
In the system above, $\Omega$ is a bounded domain of $\R^2$, with smooth boundary, and $B\in \mathcal{L}(U,X)$ is a linear operator, \textcolor{black}{as for instance $B = \mathbf{1}_{\omega}$ where $\omega \subset \Omega$, with $U = [\L^2(\omega)]^2$,
 and $\mathbf{1}_{\omega}$ denotes the extension-by-zero operator, i.e. $\mathbf{1}_{\omega}= u(x)$ for $x\in \omega$ and $(Bu)(x)= 0$ otherwise. We consider the solenoidal spaces
 \begin{eqnarray*}
  Y = \left\{ y \in [\H^1_0(\Omega)]^2\mid \ \divg y = 0\right\},
  \quad X = \left\{ y \in [\L^2(\Omega)]^2\mid \ \divg y = 0, \ y\cdot \nu=0 \text{ on } \p \Omega\right\},
  \quad Y' \supset [\H^{-1}(\Omega)]^2,
 \end{eqnarray*}
 where $\nu$ is the outer normal to $\p \Omega$, and $y_0 \in Y$. It is well-known that $X$ is a closed subspace of $[\L^2(\Omega)]^2$, \cite{Temam}. Now we are prepared to state the variational formulation of the state equation, given $u \in \L^2(0,T;U)$, in which the pressure $p$ disappears:
 \begin{eqnarray}
 &  & \text{Find  $y  \in W(0,T;Y)$ such that for all $\varphi \in \L^2(0,T;Y)$ we have almost everywhere in $(0,T)$:} \nonumber \\
 &   & \left\{\begin{array}l
 \langle \dot{y} + (y\cdot \nabla)y; \varphi \rangle_{Y';Y}
 + \nu \langle \nabla y ; \nabla \varphi \rangle_{[\L^2(\Omega)]^2} = \langle B(u) ; \varphi \rangle_{Y';Y},\\
 \langle y(0);\varphi\rangle_X = \langle y_0;\varphi\rangle_X.
 \end{array} \right. \label{eq:kk20}
 \end{eqnarray}
 For details,
 %\todo[I deleted "in particular the existence of a weak solution in $W(0,T;Y)$", because the ref for that is given below by HK01 Prop2.1]
 we refer to \cite[Chapter III.3.5]{Temam}. For our example, in the context of the general framework above, we have $f(y,u) = \nu \Delta y - (y\cdot \nabla) y + P_X B u$, where $ P_X\in \mathcal{L }([\L^2(\Omega)]^2,X)$ denotes the orthogonal projection onto $X$. To be specific, we further set for $(z,v) \in X \times U$
 \begin{equation}\label{eq:kk21}
 \ell(z,v)= -\|v\|^2_U, \quad \phi_1(z) = \phi_2(z)= \|z\|^2_{\L^2(\Omega)}.
 \end{equation}
 Let us verify $\mathbf{A0}$--$\mathbf{A4}$. Assumption $\mathbf{A0}$ is clearly satisfied. Concerning~$\mathbf{A1}$,} short considerations show that first- and second-order directional differentiability and continuity of the derivatives (and hence Fr\'echet differentiability) of $f$ will follow from the continuity of the Oseen--type term $(z_1, z_2, z_3) \mapsto (z_1 \cdot \nabla)z_2 + (z_3 \cdot \nabla)z_1$ from  $[W(0,T;Y)]^3$ to $\L^2(0,T;Y')$, which will be given below. Since $\divg z_1 = \divg z_2 = \divg z_3 = 0$, we have
\begin{eqnarray*}
 (z_1 \cdot \nabla)z_2 + (z_3 \cdot \nabla)z_1 & = & \divg ( z_2 \otimes z_1 + z_1 \otimes z_3), \\
 \| (z_1 \cdot \nabla)z_2 + (z_3 \cdot \nabla)z_1 \|_{Y'} & \leq &
 \| (z_1 \cdot \nabla)z_2 + (z_3 \cdot \nabla)z_1 \|_{[\H^{-1}(\Omega)]^2} \\
 & \leq & \| z_1 \otimes z_2 + z_3 \otimes z_1 \|_{[\L^{2}(\Omega)]^2}\leq
 \| z_1 \otimes z_2\|_{[\L^{2}(\Omega)]^4} + \|z_3 \otimes z_1 \|_{[\L^{2}(\Omega)]^4}.
\end{eqnarray*}
In view of the symmetric role played by $z_1$, $z_2$ and $z_3$, we only estimate the first term. From~\cite{Grubb} (Appendix B, Proposition B.1), in dimension $2$ there exists a constant $C>0$ independent of $z_1$ and $z_2$ such that
\begin{eqnarray*}
 \| z_1 \otimes z_2\|_{[\L^{2}(\Omega)]^4} & \leq & C\| z_1\|_{[\H^{1/2}(\Omega)]^2}
 \| z_2\|_{[\H^{1/2}(\Omega)]^2}.
\end{eqnarray*}
By interpolation we deduce
\begin{eqnarray*}
 \| z_1 \otimes z_2\|_{[\L^{2}(\Omega)]^4} & \leq &
 C\| z_1\|^{1/2}_{[\L^{2}(\Omega)]^2} \| z_1\|^{1/2}_{[\H^{1}(\Omega)]^2}
 \| z_2\|^{1/2}_{[\L^{2}(\Omega)]^2}\| z_2\|^{1/2}_{[\H^{1}(\Omega)]^2}.
\end{eqnarray*}
By integrating in time, this yields
\begin{eqnarray*}
 \| z_1 \otimes z_2\|_{\L^2(0,T;[\L^{2}(\Omega)]^4)} & \leq &
 C\| z_1\|^{1/2}_{\L^{\infty}(0,T;X)}\| z_2\|^{1/2}_{\L^{\infty}(0,T;X)}
 \| z_1\|^{1/2}_{\L^{2}(0,T;Y)}
 \| z_2\|^{1/2}_{\L^{2}(0,T;Y)} \\
 & \leq & \tilde{C}\| z_1\|_{W(0,T;Y)}\| z_2\|_{W(0,T;Y)},
\end{eqnarray*}
for some constant $\tilde{C} >0$. Thus $\mathbf{A1}$ is satisfied. Assumptions~$\mathbf{A2}$ and~$\mathbf{A3}$ are due to~\cite{HinzeKK01} (Proposition~2.1 page 928, and Proposition~2.4 page 929, respectively).\\

\textcolor{black}{Now we turn to the question of existence for problem~\eqref{pbmain}. For this purpose we consider problem~\eqref{pbmain} with the specifications~\eqref{eq:kk20} and~\eqref{eq:kk21}. Let $\{(\tau_n, u_n)\} \in [0,\tau]\times \L^2(0,T;U)$ denote a maximizing sequence. Since $\|u_n\|_{\L^2(0,T;U)}^2 \leq\gamma $ for all $n$, there exists a subsequence of $\{(\tau_n, u_n)\}$, that for simplicity we denote by the same symbol, and $(\bar \tau, \bar u) \in  [0,\tau]\times \L^2(0,T;U)$ such that $\{(\tau_n, u_n)\}$  converges weakly to  $(\bar \tau, \bar u)$. Since $B\in \mathcal{L}(U,X)$, it follows that $\{ B u_n\}$ is bounded  in $X$.
 By assuming that $y_0\in Y$, it follows that $\{y(u_n)\}$ is a bounded sequence in $W(0,T;Y) \cap \L^2(0,T; [\H^2(\Omega)]^2) \cap \H^1(0,T;X)$, see \cite[Theorem III. 3.10]{Temam} and \cite[Proposition 2.3]{HinzeKK01}. It is now standard to pass to the limit in \eqref{eq:kk20} with $u$ replaced by $u_n$ and to find that $y(u_n)$ converges at least weakly in $W(0,T;Y)$ to the solution $y(\bar u)$ associated with the control $\bar u$. The extra boundedness of $\{y(u_n)\}$ in $\L^2(0,T; [\H^2(\Omega)]^2) \cap \H^1(0,T;X)$ is needed to pass to the $\limsup$ in the functional
 \begin{equation}\label{eq:kk22}
 -\int_0^T \|u_n(t)\|^2_U \, dt + \|y(\tau_n;u_n)\|^2_{[\L^2(\Omega)]^2}   + \|y(T;u_n)\|^2_{[\L^2(\Omega)]^2}.
 \end{equation}
 From \cite[Corollary 8]{Simon}, the embedding of $\L^2(0,T; [\H^2(\Omega)]^2) \cap \H^1(0,T;X)$ in $C([0,T]; [[\H^2(\Omega)]^2,X]_{1-\epsilon})$ is relatively compact, where $\epsilon \in (0,\frac{1}{2} )$, and hence relatively compactly into $C([0,T];[\L^2(\Omega)]^2)$. Here $[[\H^2(\Omega)]^2,X]_{1-\epsilon}$ denotes the interpolation space with index $1-\epsilon$, which is well-defined since the Stokes operator associated with~\eqref{eq:kk20} in the space $X$ generates an analytic semi-group, see e.g. \cite[Theorem 2]{Giga}. In particular, possibly after taking another subsequence, we have that
 $y(u_n)  \to y(\bar u)$ in $C([0,T]; [\L^2(\Omega)]^2)$, see \cite[Lemma 1]{Simon}. Now we have the estimate
 \begin{eqnarray*}
  \|y(\tau_n;u_n)-  y(\bar \tau;\bar u)\|_{[\L^2(\Omega)]^2} & \le &
  \|y(\tau_n;u_n)- y(\tau_n;\bar u)\|_{[\L^2(\Omega)]^2} +
  \|y(\tau_n;\bar u)-  y(\bar \tau;\bar u)\|_{[\L^2(\Omega)]^2}\\
  & \le & \|y(\cdot;u_n)- y(\cdot;\bar u)\|_{C(0,T;[\L^2(\Omega)]^2)} +
  \|y(\tau_n;\bar u)-  y(\bar \tau;\bar u)\|_{[\L^2(\Omega)]^2}.
 \end{eqnarray*}
 Analogously we obtain that $ y(T;u_n) \to y(T;\bar u)$ in $[\L^2(\Omega)]^2$. Since $(\tau_n, u_n))$ was taken as a maximizing sequence, we can take the $\limsup$ in~\eqref{eq:kk22} and obtain
 \begin{equation*}
 -\int_0^T \|u(t)\|^2_U \, dt + \|y(\tau;u)\|^2_{[\L^2(\Omega)]^2}   + \|y(T;u)\|^2_{[\L^2(\Omega)]^2} \le -\int_0^T \|\bar u(t)\|^2_U \, dt +
 \|y(\bar\tau; \bar u)\|^2_{[\L^2(\Omega)]^2}   + \|y(T;\bar u)\|^2_{[\L^2(\Omega)]^2}
 \end{equation*}
 for all $u\in \L^2(0,T; U)$ with $\|u\|_{\L^2(0,T; U)}^2 \le \gamma$ and $\tau\in [0,T]$, and thus existence of an optimal control is established.
}

%Note that it is a-priori no obvious, that the associated values of the cost functional do not converge to $\infty$. We next exclude this situation. For this purpose we introduce the closed convex subset $U_c$ of $U$ by
%\begin{equation*}
%U_c= \{ v \in U: |v_i(x)| \le \rho, \text { for a.e. } x\in \omega, \text{ and } i=1,2 \},
%\end{equation*}
%and denote by $P_{U_c}$ the projection operator from $U$ onto $U_c$. With  $P_{U_c}$ given we define the new sequence
%\begin{equation*}
%\{(\tau_n, u_n)\}= \{(\tau_n,P_{U_c}\tilde u_n)\} \subset [0,T] \times L^2(0,T;U_c).
%\end{equation*}
%Since $[0,T] \times L^2(0,T;U_c)$ is bounded in $[0,T] \times L^2(0,T;U)$ there exists a subsequence, which for convenience we denote by the same symbols, and $(\bar \tau, \bar u) \in [0,T] \times L^2(0,T;U_c)$ such that $(\tau_n, u_n)$ converges weakly in $\R \times L^2(0,T;U)$ to $(\bar \tau, \bar u)$.  Since $L^2(0,T;U_c)$ is closed and convex it is also weakly closed and hence $\bar u \in L^2(0,T; U_c)$. Next we observe that
% \begin{equation}\label{eq:kk22}
%G(u_n(t)) = G(\tilde u_n(t))  \text{ and } |\tilde u_n(t)|_{L^2(\omega)} \le   |u(t)|_{L^2(\omega)} \text{ for almost every } t \in [0,T].
%\end{equation}
%These fact imply that $\{(\tau_n, u_n)\}$ is also a maximizing sequence.
%(Theorem~1.2 page 174, and Proposition~3.2 page 180, respectively).

\subsubsection{Additional regularity}
In other partial differential equations examples, the function space framework $W(0,T;Y)$ can be too restrictive. This is due to the fact that, in the linearized state equation given in assumption~$\mathbf{A3}$, the term $f_y(\bar{y},\bar{u})$ appears as a coefficient. Hence the solution of this linearized system may not be well-defined, unless $f$ satisfies appropriate growth bounds. For this purpose an alternative functional framework can be appropriate. For instance, let $\hat{Y} \subset Y \subset X$ be three Hilbert spaces endowed with a chain of continuous embeddings. Duality is understood with respect to $X \equiv X'$. We set
\begin{eqnarray*}
 \hat{W}(0,T;\hat{Y}) & = & \{ y\in \L^2(0,T;\hat{Y}) \mid \ \dot{y} \in \L^2(0,T;X) \}.
\end{eqnarray*}
We still have $\hat{W}(0,T;\hat{Y}) \hookrightarrow C([0,T];X)$, but in general we do not have $\hat{W}(0,T;\hat{Y}) \hookrightarrow C([0,T];Y)$, except for instance when we have the interpolation $Y = [X;\hat{Y}]_{1/2}$. When we do not have this embedding, we can add $y\in \mathcal{C}([0,T];Y)$ into the definition of $\hat{W}(0,T;\hat{Y})$ above. In applications, for a given smooth domain $\Omega$, we may think of $X = \L^2(\Omega)$, and $Y$, $\hat{Y}$ subspaces of $\H^1(\Omega)$, $\H^2(\Omega)$, respectively. The mapping $e$ introduced in the proof of Lemma~\ref{lemma_diffSys} is then modified to be
\begin{eqnarray*}
 \begin{array} {rccl}
  e: &  \hat{W}(0,2;\hat{Y})  \times \L^2(0,2;U) \times \R & \rightarrow & \L^2(0,2;X)\times Y.
 \end{array}
\end{eqnarray*}
The assumptions utilized in this proof are implied by~$\mathbf{A1}$--$\mathbf{A3}$. In order to transpose the assumptions~$\mathbf{A1}$--$\mathbf{A3}$ and to derive differentiability in the framework considered here, we introduce the mapping $\hat{f}:\hat{Y}\times U \rightarrow X$ associated with $f$ as a restriction. We would assume:
\begin{description}
 \item[$\mathbf{B1}$]The mapping $\hat{f}: W(0,T;\hat{Y}) \times \L^2(0,T;U) \times \R \rightarrow \L^2(0,T;X)$ is twice continuously Fr\'echet-differentiable.
 \item[$\mathbf{B2}$]For all $(u,\tau) \in \L^2(0,T;U)\times \R$, system~\eqref{mainsys} admits a unique solution in $\hat{W}(0,T;\hat{Y})$.
 \item[$\mathbf{B3}$]For all $\xi \in \L^2(0,T;X)$, $z_0 \in Y$, there exists a unique $z \in \hat{W}(0,T;\hat{Y})$ solution to the following system
 \begin{eqnarray*}
  \left\{ \begin{array} {l}
   \dot{z} = \hat{f}_y(\bar{y},\bar{u}).z + \xi \quad \text{on } (0,T), \\
   z(0) = z_0.
  \end{array} \right.
 \end{eqnarray*}
\end{description}
This framework, where strong regularity is considered, is needed when the nonlinearity of the mapping $f$ cannot be handled in the context of weak solutions. For instance, polynomial nonlinearities, as those occurring in the Schl\"{o}gel and FitzHugh-Nagumo systems (see~\cite{CRT}), require the notion of strong solutions, in Hilbert spaces of type $\H^s$ (see~\cite{Breiten}). \textcolor{black}{These are weak solutions which in addition lie in $\hat{W}(0,T;\hat{Y})$.}
In the context of parabolic equations, the time-dependent operator $\hat{f}_y(\bar{y},\bar{u})$ can be studied with the approaches of~\cite{Bardos} or~\cite{Troltzsch} for instance.\\

For a sake of being specific, in the rest of the paper we will keep the framework of the function space $W(0,T;Y)$, with assumptions~$\mathbf{A1}$--$\mathbf{A3}$. Analogous results  as those obtained for $W(0,T;Y)$ also hold for $\hat{W}(0,T;\hat{Y})$.

\subsection{Transformation of the problem}

The formulation of problem~\eqref{pbmain} does not enable us to derive optimality conditions. Indeed, the control is not continuous at the optimal time $\tau$, in general. Therefore, the trajectory is not differentiable at $\tau$ and the cost is not differentiable with respect to $\tau$.
This difficulty can be \textcolor{black}{circumvented} by introducing the following change of variables, $\pi(\cdot,\tau) : [0,2] \rightarrow [0,T]$, for all $\tau \in (0,T)$,
\begin{eqnarray*}
 \pi(s,\tau) & =  &
 \left\{ \begin{array} {ll}
  \tau s & \text{if } s\in [0,1], \\
  (T-\tau) s + 2\tau -T & \text{if } s\in [1, 2].
 \end{array} \right.
\end{eqnarray*}
Observe that $\pi(1,\tau)= \tau$. For future reference we introduce the time-derivative $\dot{\pi}$ of $\pi$ (with respect to $s$), as well its the partial derivative $\dot{\pi}_\tau$ with respect to $\tau$:
\begin{equation*}
\dot{\pi}(s,\tau)  =
\left\{ \begin{array} {ll}
\tau  & \text{if } s\in [0,1), \\
(T-\tau) & \text{if } s\in (1, 2],
\end{array} \right.
\qquad
\dot{\pi}_\tau(s,\tau)  =
\left\{ \begin{array} {ll}
1  & \text{if } s\in [0,1), \\
-1 & \text{if } s\in (1, 2].
\end{array} \right.
\end{equation*}
Observe that $\dot{\pi}_\tau$ is actually independent of $\tau$. To \textcolor{black}{simplify} the \textcolor{black}{notation}, we will simply write $\dot{\pi}_\tau(s)$. Given $\pi(\cdot,\tau)$ we introduce the following change of unknowns
\begin{eqnarray}
y:s \mapsto \tilde{y} \circ \pi(s,\tau), \quad u:s \mapsto \tilde{u} \circ \pi(s,\tau),
\quad s \in [0,2]. \label{eqChangeOfVar}
\end{eqnarray}
Then, for $(u,\tau) \in \L^2(0,2;U) \times (0,T)$, we are lead to consider the following system
\begin{eqnarray} \label{mainsys}
\left\{ \begin{array} {l}
\dot{y} = \dot{\pi}(\cdot,\tau) f(y,u) \quad \text{on } (0,2),\\
y(0) = y_0,
\end{array} \right.
\end{eqnarray}
and the following reformulated optimal control problem:
\begin{eqnarray}
& & \left\{ \begin{array} {l}
\displaystyle \max_{\tau\in(0,T), \ u \in \L^2(0,2;U)} \ J(u,\tau):= \int_0^2 \dot{\pi}(s,\tau) \ell(y(s),u(s)) \d s + \phi_1(y(1)) + \phi_2(y(2)) \vspace{2mm} \\
\displaystyle \text{subject to: }\quad
\dot{y} = \dot{\pi}(\cdot,\tau) f(y,u) \text{ in } Y',\quad y(0) = y_0 \text{ in } X, \\[1mm]
\qquad \qquad \qquad \textcolor{black}{G(u,\tau) := {\displaystyle \int_0^T } \dot{\pi}(s,\tau) \| u(s) \|_U^2 \d s- \gamma \leq 0.}
\end{array} \right. \label{pbmain2}
\end{eqnarray}

\textcolor{black}{
 Using the fact that $\pi(\cdot,\tau) \in \W^{1,\infty}(0,2;\R)$, we can verify that assumptions~$\mathbf{A0}$--$\mathbf{A3}$ imply the following ones.
 \begin{description} 
  \item[$\mathbf{A0'}$] The mapping $\ell :  W(0,2;Y) \times \L^2(0,2;U) \rightarrow  \L^1(0,2;\R)$ is twice continuously Fr\'echet-differentiable.
  \item[$\mathbf{A1'}$] The mapping $f : W(0,2;Y) \times \L^2(0,2;U) \rightarrow \L^2(0,2;Y')$ is twice continuously Fr\'echet-differentiable.
  \item[$\mathbf{A2'}$] For all $(u,\tau) \in \L^2(0,2;U)\times (0,T)$, $y_0 \in X$, system~\eqref{mainsys} admits a unique weak solution in $W(0,2;Y)$.
  \item[$\mathbf{A3'}$]For all $(y,u) \in W(0,2;Y) \times \L^2(0,2;U)$, $\xi \in \L^2(0,2;Y')$, $z_0\in X$, there exists a unique $z \in W(0,2;Y) $ solution to the following system
  \begin{eqnarray*}
   \left\{ \begin{array} {l}
    \dot{z} = \dot{\pi}(\cdot,\tau)f_{y}(y,u).z + \xi \quad \text{on } (0,2), \\
    z(0) = z_0.
   \end{array} \right.
  \end{eqnarray*}
 \end{description}
}
\textcolor{black}{
 Observe that the mapping $G:\L^2(0,2;U) \times (0,T) \rightarrow \R$ modeling the constraint is twice continuously Fr\'echet-differentiable.
}

\textcolor{black}{
 The change of unknown~\eqref{eqChangeOfVar} modifies the nature of the optimal control problem under investigation: While \eqref{pbmain} looks like a shape optimization problem with variable domains $(0,\tau)$ and $(\tau,T)$, problem \eqref{pbmain2} has the nature of a parametric optimization problem.
}

The equivalence of problems~\eqref{pbmain} and~\eqref{pbmain2} is straightforward because, on one hand, the time-derivative of $\tilde{y} \circ \pi(\cdot,\tau)$ is expressed with the chain rule as $\dot{\pi}(\cdot,\tau) \dot{\tilde{y}}  \circ \pi(\cdot,\tau) = \dot{\pi}(\cdot,\tau)f(\tilde{y}\circ \pi(\cdot,\tau), \tilde{u}\circ \pi(\cdot,\tau))$. On the other hand, the integral on $(0,T)$ can be split on $(0,\tau) \cup (\tau ,T)$, and the change of variables $\pi(\cdot,\tau)$ is used for transforming the integrals on $(0,\tau)$ and $(\tau,T)$ into the integrals on $(0,1)$ and $(1,2)$, respectively.
To sum up, if $\tilde{u} \in \L^2(0,T;U)$, $\tau \in \text{\textcolor{black}{$(0,T)$}}$ and $u= \tilde{u}\circ \pi(\cdot,\tau)$, the pair $(\tilde{u},\tau)$ is an optimal solution of the original problem~\eqref{pbmain} if and only if the pair $(u,\tau)$ is an optimal solution of the reformulated problem~\eqref{pbmain2}.

\section{First and second-order optimality conditions} \label{sec3}

\textcolor{black}{In this section, we provide a first and second-order sensitivity analysis for the reformulated problem~\eqref{pbmain2}. Due to the presence of the parameter $\tau$ in addition to the control variable $u$, it is non-standard to obtain elegant representations for the first- and second-order derivatives. Once they are given, their expressions are used for obtaining necessary and sufficient optimality conditions, as well as for iterative numerical methods. }

Throughout this section, $\bar{u} \in \L^2(0,2;U)$ is a fixed value of the control and $\bar{\tau} \in (0,T)$ a fixed value of the variable $\tau$. Note that we do not follow up the special cases $\bar{\tau}=0$ and $\bar{\tau}=T$. We denote by $\bar{y} \in W(0,2;Y)$ the corresponding state variable, which is the solution of~\eqref{mainsys} for $(u,\tau)=(\bar{u}, \bar{\tau})$.

The approach we use for deriving optimality conditions is classical, as described in \cite{HinzeKK01, Hinze}, for instance. However, due to the time transformation and the additional optimization variable, special attention is required. Considering the state equation as a constraint of the optimization problem~\eqref{pbmain2}, our approach mainly consists in computing the first- and second-order derivatives of the associated Lagrangian.

\subsection{Linearization of the system}

We introduce the control-to-state mapping \textcolor{black}{$(u,\tau) \in \L^2(0,2;U) \times (0,T) \mapsto S(u,\tau) \in W(0,2;Y)$}, where $S(u,\tau)$ is the solution of~\eqref{mainsys}. Remember that $\bar{y} = S(\bar{u},\bar{\tau})$.
The differentiability properties of the mapping $S$ derive from assumptions~$\mathbf{A1'}$--$\mathbf{A3'}$.

\begin{lemma} \label{lemma_diffSys}
 The mapping $S$ is twice continuously Fr\'echet-differentiable on $\L^2(0,2;U) \times (0,T)$. For $v \in \L^2(0,2;U)$, the derivatives $z=S_u(\bar{u},\bar{\tau}).v $ and $w= S_\tau(\bar{u},\bar{\tau})$ are the \textcolor{black}{respective solutions -- in the weak sense -- of} the following systems:
 \begin{equation*} \label{syssensuv}
 \left\{ \begin{array} {l}
 \dot{z} = \dot{\pi}(\cdot,\bar{\tau})f_y(\bar{y},\bar{u}).z + \dot{\pi}f_u(\bar{y},\bar{u}).v \quad \text{on } (0,2), \\
 z(0) = 0,
 \end{array} \right.
 \quad
 \left\{ \begin{array} {l}
 \dot{w} = \dot{\pi}(\cdot,\bar{\tau})f_y(\bar{y},\bar{u}).w + \dot{\pi}_{\tau}f(\bar{y},\bar{u}) \quad \text{on } (0,2), \\
 w(0) = 0.
 \end{array} \right.
 \end{equation*}
\end{lemma}

\begin{proof}
 Consider the mapping
 \begin{eqnarray*}
  \begin{array} {rccl}
   e: & W(0,2;Y) \times \L^2(0,2;U) \times (0,T) & \rightarrow & \L^2(0,2;Y')\times X, \\
   & (y,u,\tau) & \mapsto & (\dot{y} - \dot{\pi}(\cdot,\tau)f(y,u) , y(0) - y_0).
  \end{array}
 \end{eqnarray*}
 Since we have the identity $e(S(u,\tau),u,\tau) = 0$, assumptions~$\mathbf{A1'}$--$\mathbf{A3'}$ enable us to apply the implicit function theorem\textcolor{black}{, in the same fashion as~\cite{Hinze}, section~1.6, pages~57-58}. In fact, \textcolor{black}{invertibility} of the mapping $e_y(y,u,\tau)$ is a consequence of $\mathbf{A3'}$, and the required smoothness conditions follow from $\mathbf{A1'}$. The result then follows.
\end{proof}

\color{black}
\begin{corollary}
 The cost function $J$ is twice continuously Fr\'echet differentiable.
\end{corollary}

\begin{proof}
 The following mapping is twice continuously Fr\'echet differentiable
 \begin{eqnarray*}
  (y,u,\tau) \in W(0,2;Y) \times \L^2(0,2;U) \times (0,T)
  & \mapsto &
  \int_0^2 \dot{\pi}(s,\tau) \ell(y(s),u(s)) \d s + \phi_1(y(1)) + \phi_2(y(2)),
 \end{eqnarray*}
 by $\mathbf{A0'}$, and because $\phi_1$ and $\phi_2$ are twice continuously Fr\'echet differentiable. Therefore, $J$ is also twice continuously Fr\'echet differentiable, by composition by $S$.
\end{proof}

\color{black}

\subsection{Vectorial formalism}
We define the functional space $\mathcal{Y}$ and its dual space as follows:
\begin{eqnarray*}
 \mathcal{Y}= X \times X \times \L^2(0,2;Y), & &
 \mathcal{Y}'= X \times X \times \L^2(0,2;Y').
\end{eqnarray*}
Next we introduce the mapping $\mathbf{S}$ as follows:
\begin{eqnarray*}
 \mathbf{S}: (u,\tau) \in \L^2(0,2;U) \times (0,T) & \mapsto & (S(u,\tau)(1), S(u,\tau)(2), S(u,\tau)) \in \mathcal{Y}.
\end{eqnarray*}
As a consequence of Lemma~\ref{lemma_diffSys}, the mapping $\mathbf{S}$ is twice continuously differentiable. Its first-order derivatives are given by:
\begin{eqnarray} \label{eqDerivativesS}
\mathbf{S}_u(u,\tau)= \big(S_u(u,\tau)(1), S_u(u,\tau)(2), S_u(u,\tau) \big), & &
\mathbf{S}_\tau(u,\tau)= \big(S_\tau(u,\tau)(1), S_\tau(u,\tau)(2), S_\tau(u,\tau) \big).
\end{eqnarray}
Let us define the operator $\mathcal{K} \in \mathcal{L}\left(\L^2(0,2;Z);\mathcal{Y} \right)$ by
\begin{eqnarray*}
 \begin{array} {rrcl}
  \mathcal{K} : &
  \xi & \mapsto & (z(1),z(2),z)
 \end{array}
\end{eqnarray*}
where $z \in W(0,2;Y)$ is defined -- in virtue of assumption~$\mathbf{A3}$ -- as the solution of
\begin{eqnarray}
\left\{ \begin{array} {l}
\dot{z} = \dot{\pi}(\cdot,\bar{\tau})f_y(\bar{y},\bar{u}).z + \xi \quad \text{on } (0,2), \\
z(0) = 0.
\end{array} \right. \label{syslemmaS}
\end{eqnarray}

\begin{lemma} \label{lemma-adjS}
 The adjoint $\mathcal{K}^\ast \in \mathcal{L} \left( \mathcal{Y}' ; \L^2(0,2;Z')\right)$ of $\mathcal{K}$ is given by
 $\mathcal{K}^{\ast}(a,b,w) = q$,
 where $q$ is the solution of
 \begin{eqnarray}
 \left\{ \begin{array} {l}
 -\dot{q} = \dot{\pi}(\cdot,\bar{\tau})f_y(\bar{y},\bar{u})^\ast.q  + w \quad \text{on } (0,1)\cup (1,2), \\
 q(2) = b,\\
 q(1^+)-q(1^-) + a = 0.
 \end{array} \right. \label{syslemmaStar}
 \end{eqnarray}
\end{lemma}

%\begin{remark}
%\textcolor{black}{The evaluations of the adjoint operators $f_y(y,u)^*$ and $f_u(y,u)^*$ at point $v$ have to be understood pointwise, that is to say: $\left[f_y(y(\cdot),u(\cdot))^*.v(\cdot)\right](t) = f_y(y(t),u(t))^*.v(t)$.}
%\end{remark}

\begin{proof}
 Let $\xi \in \L^2(0,2;Z)$ and let $z$ be the solution of system~\eqref{syslemmaS} corresponding to $\xi$. Let $(a,b,w) \in \mathcal{Y}'$ and denote by $q$ the solution of system~\eqref{syslemmaStar} corresponding to $(a,b,w)$. We calculate by integration by parts
 \begin{eqnarray*}
  \langle (a,b,w);\mathcal{K}(\xi)\rangle_{\mathcal{Y}';\mathcal{Y}} & = & \langle w;z \rangle_{\L^2(0,2;Y');\L^2(0,2;Y)}
  + \langle a;z(1)\rangle_{X} + \langle b;z(2)\rangle_{X} \\
  & = & \int_0^{1} \langle -\dot{q} -\dot{\pi} f_y^\ast(\bar{y},\bar{u}).q; z\rangle_{Y';Y}\d s + \int_{1}^2 \langle -\dot{q} -\dot{\pi} f_y^\ast(\bar{y},\bar{u}).q; z\rangle_{Y';Y}\d s \\
  & &  + \, \langle a;z(1)\rangle_{X} + \langle b;z(2)\rangle_{X}\\
  & = & \int_0^{1} \langle q;\dot{z}-\dot{\pi} f_y(\bar{y},\bar{u}).z \rangle_{Z';Z}\d s -\langle q(1^-);z(1)\rangle_X  + \int_{1}^2 \langle q;\dot{z}-\dot{\pi} f_y(\bar{y},\bar{u}).z \rangle_{Z';Z}\d s\\
  & &   + \, \langle q(1^+);z(1)\rangle_X - \langle q(2);z(2)\rangle_X + \langle a;z(1)\rangle_X + \langle b;z(2)\rangle_X \\
  & = & \int_0^{1} \langle q; \xi \rangle_{Z';Z} \d s  + \int_{1}^2 \langle q;\xi \rangle_{Z';Z}\d s,
 \end{eqnarray*}
 which leads to $\langle (a,b,w);\mathcal{K}(\xi)\rangle_{\mathcal{Y}';\mathcal{Y}} = \langle q ; \xi \rangle_{\L^2(0,2;Z');\L^2(0,2;Z)}$ and thus completes the proof.
\end{proof}

Lemma~\ref{lemma-adjS} enables us to conveniently express the adjoint operators of $\mathbf{S}_u(\bar{u},\bar{\tau})$ and $\mathbf{S}_\tau(\bar{u},\bar{\tau})$.

\begin{corollary}
 The adjoint operators $\mathbf{S}_u(\bar{u},\bar{\tau})^{\ast} \in \mathcal{L}\left(\mathcal{Y}';\L^2(0,2;U')\right)$ and $\mathbf{S}_\tau(\bar{u},\bar{\tau})^{\ast} \in \mathcal{L}\left(\mathcal{Y}';\R\right)$ are given by
 \begin{eqnarray*}
  \mathbf{S}_u(\bar{u},\bar{\tau})^{\ast}.(a,b,w)  =  \dot{\pi}(\cdot,\bar{\tau}) f_u^\ast(\bar{y},\bar{u}) \mathcal{K}^\ast(a,b,w),
  & &
  \mathbf{S}_\tau(\bar{u},\bar{\tau})^{\ast}.(a,b,w)  =  \int_0^2\dot{\pi}_\tau \langle f(\bar{y},\bar{u}) ;\mathcal{K}^\ast(a,b,w)\rangle_{Y';Y} \d s.
 \end{eqnarray*}
\end{corollary}

\begin{proof}
 As a consequence of Lemma~\ref{lemma_diffSys} and~\eqref{eqDerivativesS}, we have
 $\mathbf{S}_u(\bar{u},\bar{\tau}) = \dot{\pi}(\cdot,\bar{\tau}) \mathcal{K} \circ  f_u(\bar{y},\bar{u})$ and $\mathbf{S}_\tau(\bar{u},\bar{\tau}) = \dot{\pi}_\tau \mathcal{K} \circ f(\bar{y},\bar{u})$.
 The result now follows from Lemma~\ref{lemma-adjS}.
\end{proof}

\subsection{Lagrangian formulation and computation of derivatives}

We recall that $\bar{u} \in \L^2(0,2;U)$ and $\bar{\tau}$ are fixed values of the control and the time variable, respectively. We also fix $\bar{y}= S(\bar{u},\bar{\tau})$ and $\bar{\mathbf{y}}= \mathbf{S}(\bar{u},\bar{\tau})= (\bar{y}(1),\bar{y}(2),\bar{y})$.
We introduce the Hamiltonian:
\begin{eqnarray*}
 \begin{array}{rrcl}
  H:& Y \times U \times Z' \times \R & \mapsto & \mathbb{R} \\
  & (y,u,p,\lambda) & \rightarrow & \ell(y,u) + \langle p; f(y,u) \rangle_{Z';Z} \textcolor{black}{ - \lambda \| u \|_U^2.}
 \end{array}
\end{eqnarray*}
The adjoint state is defined as the solution of the following linear system:
\begin{eqnarray} \label{sysadj}
\left\{ \begin{array} {l}
-\dot{p} = \dot{\pi}(\cdot,\bar{\tau})H_y(\bar{y}, \bar{u}, p)\quad \text{on } (0,1) \cup (1,2),\\
p(2) = D \phi_2(\bar{y}(2)) , \\
p(1^+) - p(1^-) + D\phi_1(\bar{y}(1)) = 0.
\end{array} \right.
\end{eqnarray}
It satisfies:
\begin{eqnarray*}
 \bar{p}_{|(0,1)} \in W(0,1;Z'), & & \bar{p}_{|(1,2)} \in W(1,2;Z').
\end{eqnarray*}
\textcolor{black}{
 In the system above, the variable $\lambda$ does not appear, since $H_y$ is independent of $\lambda$.} Recall the continuous embeddings $W(I;Z') \hookrightarrow \mathcal{C}(\overline{I};X)$, for $I = (0,1)$ and $I = (1,2)$. In order to solve this backward system, we first consider $\bar{p}(2) = D\phi_2(\bar{y}(2))$ as the initial condition in $X$, next compute $\bar{p}$ on $(1,2)$ according to the first equation of~\eqref{sysadj}, deduce $\bar{p}(1^-)$ from $\bar{p}(1^+)$ with the transmission condition in $X$, and finally compute $\bar{p}$ on $(0,1)$ as previously. From a more abstract point of view, the affine mapping $p \mapsto H_y(\bar{y},\bar{u},p) = f_y(\bar{y},\bar{u})^*.p + \ell_y(\bar{y},\bar{u})$ is in the form of the right-hand-side in system~\eqref{syslemmaStar}, and thus Lemma~\ref{lemma-adjS} allows the existence and uniqueness of a solution to system~\eqref{sysadj}.

\textcolor{black}{
 We work with two Lagrangian functionals. The first one is defined as follows:
 \begin{eqnarray*}
  \begin{array}{rrcl}
   L\colon & \L^2(0,2;U) \times (0,T) \times \R & \rightarrow & \R \\
   & (u,\tau,\lambda) & \mapsto & J(u,\tau) - \lambda G(u,\tau).
  \end{array}
 \end{eqnarray*}
 We also consider the following extended Lagrangian:
 \begin{eqnarray*}
  \begin{array} {rl}
   \mathbf{L} : &
   \big(X\times X\times W(0,2;Y)\big) \times \L^2(0,2;U) \times (0,T) \times W(0,2;Y) \times \R \rightarrow  \mathbb{R} \\[2mm]
   & \big( \mathbf{y}= (a_1,a_2,y), u, \tau, p, \lambda \big)  \mapsto
   \phi_1(a_1)+ \phi_2(a_2) + \displaystyle{\int_0^2} \Big(\dot{\pi}(s,\tau) H(y,u,p,\lambda)(s)-\langle p(s);\dot{y}(s) \rangle_{Z';Z} \Big)\d s \\[4mm]
   & \hspace*{125pt}
   - \langle p(0);y(0)-y_0 \rangle_X
   + \langle p(2);y(2)-a_2 \rangle_X
   - \langle [p]_{1}; y(1)-a_1 \rangle_X,
  \end{array}
 \end{eqnarray*}
 where we denote $[p]_{1} = p(1^+) - p(1^-)$.
 Note that the functionals $L$ and $\mathbf{L}$ are both twice continuously differentiable. The following lemma enables us to derive the first and second-order derivatives of the $L$ in a convenient way.
}

\begin{lemma} \label{lemma_Ly}
 The following identity holds:
 \textcolor{black}{
  \begin{equation} \label{lagrangianIsCost}
  L(u,\tau,\lambda)= \mathbf{L}(\mathbf{S}(u,\tau),u,\tau,p,\lambda), \quad
  \forall (u,\tau,p,\lambda) \in \L^2(0,2;U)\times (0,T)\times W(0,2;Z') \times \R.
  \end{equation}
 }
 Moreover, for all $\lambda \in \R$ we have
 \begin{eqnarray*}
  \mathbf{L}_{\mathbf{y}}(\bar{\mathbf{y}},\bar{u},\bar{\tau}, \bar{p}, \lambda)  =  0
  & &  \text{in $\mathcal{Y}'$}.
 \end{eqnarray*}
\end{lemma}

\begin{proof}
 Identity~\eqref{lagrangianIsCost} follows directly from the definitions of $J$, $\mathbf{S}$, $L$, and $\mathbf{L}$. We decompose $\mathbf{L}_{\mathbf{y}}(\bar{\mathbf{y}},\bar{u},\bar{\tau},\bar{p},\lambda) \in \mathcal{Y}'$ into $\mathbf{L}_{a_1}(\bar{\mathbf{y}},\bar{u},\bar{\tau},\bar{p},\lambda) \in X$, $\mathbf{L}_{a_2}(\bar{\mathbf{y}},\bar{u},\bar{\tau},\bar{p},\lambda) \in X$ and $\mathbf{L}_{y}(\bar{\mathbf{y}},\bar{u},\bar{\tau},\bar{p},\lambda)\in Y'$. From the definition of the adjoint state, we obtain the following identities in $X$:
 \begin{eqnarray*}
  \mathbf{L}_{a_1}(\bar{\mathbf{y}},\bar{u},\bar{\tau},\bar{p},\lambda)= D\phi_1(\bar{y}(1)) + [\bar{p}]_{1} =0, & &
  \mathbf{L}_{a_2}(\bar{\mathbf{y}},\bar{u},\bar{\tau},\bar{p},\lambda)= D\phi_2(\bar{y}(2)) - \bar{p}(2)= 0.
 \end{eqnarray*}
 Moreover, for all $\delta y \in \W(0,2;Y)$, we obtain by integration by parts
 \begin{eqnarray*}
  \langle \mathbf{L}_y(\bar{\mathbf{y}},\bar{u},\bar{\tau},\bar{p},\lambda); \delta y) \rangle_{Y';Y} & = &
  \int_0^2 \dot{\pi}(s,\bar{\tau})\langle H_y(\bar{y},\bar{u},\bar{p})(s) ;\delta y(s)\rangle_{Y';Y}\d s
  - \int_0^2 \langle \bar{p}(s);\delta \dot{y}(s) \rangle_{Y;Y'}\d s \\
  & & \qquad -\, \langle \bar{p}(0); \delta y(0) \rangle_X
  + \langle \bar{p}(2); \delta y(2) \rangle_X
  - \langle [\bar{p}]_{1}; \delta y(1) \rangle_X \\
  & =& 0,
 \end{eqnarray*}
 which concludes the proof.
\end{proof}

\begin{proposition} \label{propJ1}
 The first-order derivatives of $L$ are given by:
 \begin{align}
 {L}_u(\bar{u},\bar{\tau},\lambda).v = \ & \mathbf{L}_u(\bar{\mathbf{y}},\bar{u},\bar{\tau},\bar{p},\lambda).v
 = \int_0^2 \dot{\pi}(s,\bar{\tau}) H_u(\bar{y},\bar{u},\bar{p},\lambda)(s).v(s) \d s,
 \label{eqDerivativesL1} \\
 {L}_\tau(\bar{u},{\bar\tau},\lambda) = \ & \mathbf{L}_\tau(\bar{\mathbf{y}},\bar{u},\bar{\tau},\bar{p}) =
 \int_0^2 \dot{\pi}_\tau H(\bar{y},\bar{u},\bar{p},\lambda)(s) \d s. \label{eqDerivativesL2}
 \end{align}
\end{proposition}

\begin{proof}
 The result follows directly from Lemma~\ref{lemma_Ly}. Applying the chain rule to~\eqref{lagrangianIsCost}, we obtain:
 \begin{eqnarray*}
  D{L}(\bar{u},\bar{\tau},\lambda) & = & \begin{pmatrix} \mathbf{L}_{\mathbf{y}}(\bar{\mathbf{y}},\bar{u},\bar{\tau},\bar{p},\lambda)\mathbf{S}_u(\bar{u},\bar{\tau}) +
   \mathbf{L}_u(\bar{\mathbf{y}},\bar{u},\bar{\tau},\bar{p},\lambda) \\
   \mathbf{L}_{\mathbf{y}}(\bar{\mathbf{y}},\bar{u},\bar{\tau},\bar{p},\lambda)\mathbf{S}_\tau(\bar{u},\bar{\tau})  +
   \mathbf{L}_\tau(\bar{\mathbf{y}},\bar{u},\bar{\tau},\bar{p},\lambda) \end{pmatrix}.
  %DL(\bar{\mathbf{y}},\bar{u},\bar{\tau})\begin{pmatrix} \mathbf{S}_u(\bar{u},\bar{\tau}) & \mathbf{S}_\tau(\bar{u},\bar{\tau}) \\ \I & 0 \\ 0 & 1 \end{pmatrix}.
 \end{eqnarray*}
 Since from Lemma~\ref{lemma_Ly} we have $\mathbf{L}_{\mathbf{y}}(\bar{\mathbf{y}},\bar{u},\bar{\tau},\bar{p},\lambda)=0$, formulas \eqref{eqDerivativesL1} and \eqref{eqDerivativesL2} hold.
\end{proof}

In the following proposition, we calculate the Hessian of $L$ easily, thanks to the Lagrangian formalism described above. The fact that $\mathbf{L}_{\mathbf{y}}(\bar{\mathbf{y}},\bar{u},\bar{\tau}, \bar{p}, \lambda)= 0$ is a key property here. We denote by $\I$ the linear identity mapping in $\L^2(0,2;U')$.

\begin{proposition} \label{propJ2}
 The second-order derivative of ${L}$ (with respect to $(u,\tau)$) is given by:
 \begin{eqnarray*}
  D^2 {L}(\bar{u},\bar{\tau},\lambda) & = &
  \begin{pmatrix}
   \mathbf{S}_u^\ast(\bar{u},\bar{\tau}) & \I & 0 \\
   \mathbf{S}_\tau^\ast(\bar{u},\bar{\tau}) & 0 & 1
  \end{pmatrix}
  D^2 \mathbf{L}(\bar{\mathbf{y}},\bar{u},\bar{\tau},\bar{p},\lambda)
  \begin{pmatrix}
   \mathbf{S}_u(\bar{u},\bar{\tau}) & \mathbf{S}_\tau(\bar{u},\bar{\tau}) \\
   \I & 0 \\
   0 & 1
  \end{pmatrix}.
 \end{eqnarray*}
 The second-order derivatives read, in a more explicit form, as
 \begin{eqnarray*}
  D^2 L(\bar{u},\bar{\tau},\lambda).\big( (v,\theta), (\hat{v},\hat{\theta}) \big)
  & = & D^2 \phi_1(\bar{y}(1)). \big( z(1), \hat{z}(1) \big)
  + D^2 \phi_2(\bar{y}(2)). \big( z(2), \hat{z}(2) \big) \\
  & & + \int_0^2 \dot{\pi}(s,\bar{\tau}) D^2 H(\bar{y},\bar{u},\bar{p},\lambda).
  \big( (z,v),(\hat{z},\hat{v}) \big)(s) \d s \\
  & & + \theta \int_0^2 \dot{\pi}_\tau (s) DH(\bar{y},\bar{u},\bar{p},\lambda)\big( \hat{z}, \hat{v} \big)(s) \d s
  +\hat{\theta} \int_0^2 \dot{\pi}_\tau (s) DH(\bar{y},\bar{u},\bar{p},\lambda)\big( z, v \big)(s) \d s,
 \end{eqnarray*}
 where $z= S_u(\bar{u},\bar{\tau}).v+ S_\tau(\bar{u},\bar{\tau}).\theta$ and $\hat{z}= S_u(\bar{u},\bar{\tau}).\hat{v}+ S_\tau(\bar{u},\bar{\tau}).\hat{\theta}$.
\end{proposition}

\begin{proof}
 Once again, the proposition follows directly from Lemma~\ref{lemma_Ly}. Applying the chain rule to~\eqref{lagrangianIsCost}, we obtain:
 \begin{eqnarray*}
  D^2{L}(\bar{u},\bar{\tau},\lambda) & = &
  \begin{pmatrix}
   \mathbf{S}_u^\ast(\bar{u},\bar{\tau}) & \I & 0 \\
   \mathbf{S}_\tau^\ast(\bar{u},\bar{\tau}) & 0 & 1
  \end{pmatrix}
  D^2 \mathbf{L}(\bar{\mathbf{y}},\bar{u},\bar{\tau},\bar{p})
  \begin{pmatrix}
   \mathbf{S}_u(\bar{u},\bar{\tau}) & \mathbf{S}_\tau(\bar{u},\bar{\tau}) \\
   \I & 0 \\
   0 & 1
  \end{pmatrix}
  +
  \mathbf{L}_{\mathbf{y}}(\bar{\mathbf{y}},\bar{u},\bar{\tau})  D^2 \mathbf{S}(\bar{u},\bar{\tau}).
 \end{eqnarray*}
 The term involving $D^2\mathbf{S}(\bar{u},\bar{\tau})$ vanishes, since $\mathbf{L}_{\mathbf{y}}(\bar{\mathbf{y}},\bar{u},\bar{\tau},\bar{p})= 0$. The explicit form follows from the compact relation below, where the notation of the variable $(\bar{\mathbf{y}},\bar{u},\bar{\tau},\bar{p})$ has been omitted, for a sake of clarity:
 \begin{eqnarray*}
  D^2 \mathbf{L} & = &
  \left(
  \begin{array}{cc|ccc}
   D^2 \phi_1(\bar{y}(1)) & 0 & 0 & 0 & 0 \\
   0 & D^2 \phi_2(\bar{y}(2)) & 0 & 0 & 0 \\ \hline
   0 & 0 & \mathbf{L}_{yy} & \mathbf{L}_{yu} & \mathbf{L}_{y\tau} \\
   0 & 0 & \mathbf{L}_{uy} & \mathbf{L}_{uu} & \mathbf{L}_{u\tau} \\
   0 & 0 & \mathbf{L}_{\tau y} & \mathbf{L}_{\tau u} & 0
  \end{array}
  \right).
 \end{eqnarray*}
 Denoting $\mathbf{z} = (z(1),z(2),z)$ and $\hat{\mathbf{z}} = (\hat{z}(1),\hat{z}(2),\hat{z})$, the partial derivatives above are formally given by
 \begin{eqnarray*}
  \left\langle
  \begin{pmatrix}
   \mathbf{L}_{yy} & \mathbf{L}_{yu} \\
   \mathbf{L}_{uy} & \mathbf{L}_{uu}
  \end{pmatrix} .
  \begin{pmatrix}
   z \\
   v
  \end{pmatrix} ; \begin{pmatrix} \hat{z} \\ \hat{v} \end{pmatrix}  \right\rangle_{\L^2(0,2;Y)';\L^2(0,2;Y)}
  & = & \int_0^2 \dot{\pi}(s,\bar{\tau}) D^2 H(\bar{y},\bar{u},\bar{p},\lambda)(s).
  \big( (z,v),(\hat{z},\hat{v}) \big) (s) \d s ,\\
  \left\langle
  \begin{pmatrix}
   \mathbf{L}_{\tau y} \\
   \mathbf{L}_{\tau u}
  \end{pmatrix};
  \begin{pmatrix}
   z \\
   v
  \end{pmatrix}
  \right\rangle_{\L^2(0,2;Y)';\L^2(0,2;Y)}
  & = & \int_0^2 \dot{\pi}_\tau DH(\bar{y},\bar{u},\bar{p},\lambda)(s). ( z, v ) (s) \d s.
 \end{eqnarray*}
 So the proof is complete.
\end{proof} 

\subsection{Optimality conditions}

\textcolor{black}{
 We give in this subsection necessary and sufficient optimality conditions. They involve the first and second-order derivatives of $L$, which have been calculated previously.
 First we introduce some notation. For $(v,\theta) \in \L^2(0,2;U)\times \R$, we denote
 \begin{eqnarray*}
  \| (v,\theta) \| & = & \left(\| v \|_{\L^2(0,2;U)}^2 + \theta^2 \right)^{1/2}.
 \end{eqnarray*}
 Given $\varepsilon > 0$, we denote
 \begin{eqnarray*}
  B_\varepsilon(\bar{u},\bar{\tau})
  & = & \big\{ (u,\tau) \in \L^2(0,2;U) \times (0,T) \,:\, \| (u-\bar{u},\tau-\bar{\tau}) \| < \varepsilon \big\}.
 \end{eqnarray*}
}

\subsubsection{First-order optimality conditions}

\color{black}

\begin{proposition} \label{Prop3}
 If the pair $(\bar{u},\bar{\tau})$ is locally optimal, then there exists a unique $\bar \lambda \geq 0$ such that
 \begin{equation} \label{eqKKT}
 L_u(\bar{u},\bar{\tau},\bar{\lambda})= 0, \quad L_\tau(\bar{u},\bar{\tau},\bar{\lambda})= 0 \quad \text{and} \quad
 \bar{\lambda} G(\bar{u},\bar{\tau})= 0.
 \end{equation}
\end{proposition}

\begin{proof}
 If the constraint is not active (i.e.\@ $G(\bar{u},\bar{\tau}) < 0$), then the result is clearly satisfied with $\bar{\lambda}= 0$.
 If the constraint is active (i.e.\@ $G(\bar{u},\bar{\tau})= 0$), then $\bar{u} \neq 0$ and therefore $D_u G(\bar{u},\bar{\tau})$ is non-zero. As a consequence, the linear independence condition of qualification holds, which ensures the existence and uniqueness of a Lagrange multiplier $\bar{\lambda}$ satisfying the Karush-Kuhn-Tucker conditions \eqref{eqKKT}.
\end{proof}

\color{black}

\begin{remark}
 When the Hamiltonian is strongly uniformly convex with respect to $u$, the optimal control $\bar{u}$ inherits some regularity properties of $\bar{y}$ and $\bar{p}$. In the current framework, the optimal control is not continuous at time~$1$ (in general), because of the jump of the adjoint state.
\end{remark}

\subsubsection{Second-order: Case of an inactive constraint}

\begin{proposition}
 If $(\bar{u},\bar{\tau})$ is a solution to problem \eqref{pbmain2} such that $G(\bar{u},\bar{\tau}) < 0$,
 then for all $(v,\theta) \in \L^2(0,2;U)\times \R$,
 \begin{eqnarray*}
  D^2J(\bar{u},\bar{\tau}).(v,\theta)^2 & \leq & 0.
 \end{eqnarray*}
 Conversely, if $(\bar{u},\bar{\tau})$ is feasible, if $DJ(\bar{u},\bar{\tau})= 0$, and if there exists $\alpha > 0$ such that for all $(v,\theta) \in \L^2(0,2;U) \times \R$
 \begin{eqnarray} \label{eqCS2}
 D^2J(\bar{u},\bar{\tau}).(v,\theta)^2 & \leq & -\alpha \| (v,\theta) \|^2,
 \end{eqnarray}
 then $(\bar{u},\bar{\tau})$ is a local solution to problem \eqref{pbmain2}. More precisely, for all $\beta \in (0,\alpha)$, there exists $\varepsilon>0$ such that for all $(u,\tau) \in B_{\varepsilon}(\bar{u},\bar{\tau})$,
 \begin{eqnarray} \label{eqGrowth}
 J(u,\tau) & \leq & J(\bar{u},\bar{\tau}) - \frac{\beta}{2} \| (u-\bar{u}, \tau-\bar{\tau}) \|^2.
 \end{eqnarray}
\end{proposition}

\begin{proof}
 If $(\bar{u},\bar{\tau})$ is a solution to \eqref{pbmain2}, then for $\varepsilon>0$ small enough the pair $(\bar{u}+\varepsilon v,\bar{\tau}+ \varepsilon \theta)$ is feasible and therefore
 \begin{eqnarray*}
  0 & \geq & \lim_{\varepsilon \downarrow 0} \frac{J(\bar{u}+\varepsilon v, \bar{\tau}+ \varepsilon \theta)-J(\bar{u},\bar{\tau})}{\varepsilon^2}= \frac{1}{2}D^2 J(\bar{u},\bar{\tau})(v,\theta)^2,
 \end{eqnarray*}
 because $DJ(\bar{u},\bar{\tau})=0$.
 Conversely, assume that $DJ(\bar{u},\bar{\tau})=0$ and that \eqref{eqCS2} holds. Let $\varepsilon>0$ be such that for all $(u,\tau) \in B_{\varepsilon}(\bar{u},\bar{\tau})$, we have
 \begin{eqnarray} \label{eqContHessian}
 | (D^2J(u,\tau) - D^2J(\bar{u},\bar{\tau})).(v,\theta)^2 | & \leq & (\alpha-\beta)\| (v, \theta) \|^2.
 \end{eqnarray}
 By Taylor's Theorem, for all $(u,\tau) \in B_{\varepsilon}(\bar{u},\bar{\tau})$, there exists $\mu \in [0,1]$ such that
 \begin{eqnarray*}
  J(u,\tau) - J(\bar{u},\bar{\tau}) & = & \frac{1}{2}D^2J(\bar{u}+ \mu v, \bar{\tau}+ \mu \theta).(v,\theta)^2,
 \end{eqnarray*}
 where $v= u-\bar{u}$ and $\theta= \tau-\bar{\tau}$.
 \textcolor{black}{
  We obtain
  \begin{eqnarray*}
   J(u,\tau) - J(\bar{u},\bar{\tau})
   & \leq & \frac{1}{2}D^2J(\bar{u}, \bar{\tau}).(v,\theta)^2
   + \frac{1}{2} \big| \big( D^2J(\bar{u}, \bar{\tau}) - D^2J(\bar{u}+ \mu v, \bar{\tau}+ \mu \theta) \big).(v,\theta)^2 \big| \\
   & \leq &  - \frac{1}{2}\alpha \| (v,\theta) \|^2 + \frac{1}{2}(\alpha-\beta) \| (v,\theta) \|^2
   = -\frac{1}{2} \beta \| (v,\theta) \|^2,
  \end{eqnarray*}
  which completes the proof.}
\end{proof}

\color{black}

\subsubsection{Second-order analysis: Case of an active constraint with strict complementarity.}

All along the paragraph, the pair $(\bar{u},\bar{\tau})$ is assumed to be such that the first-order necessary optimality condition of Proposition~\ref{Prop3} holds with $\bar{\lambda} > 0$ and $G(\bar{u},\bar{\tau})= 0$. We do not consider the case of a null Lagrange multiplier with an active constraint. We define
\begin{eqnarray*}
 w(\cdot)= 2 \dot{\pi}(\cdot,\bar{\tau})\bar{u}(\cdot) \in \L^2(0,2;U) & \text{and} &
 r= \int_0^2 \dot{\pi}_{\tau}(s) \| \bar{u}(s) \|_U^2 \d s,
\end{eqnarray*}
so that
\begin{equation*}
DG(\bar{u},\bar{\tau}).(v,\theta)= \langle w;v \rangle_{\L^2(0,2;U)} + r\theta.
\end{equation*}
Observe that $\bar{u} \neq 0$, because $G(\bar{u},\bar{\tau})= 0$, and thus $w \neq 0$.
We also set
\begin{equation*}
(\bar{w},\bar{r})= \frac{(w,r)}{\|(w,r) \|} \in \L^2(0,2;U) \times \R.
\end{equation*}
and define a linear form $\zeta: \L^2(0,2;U) \times \R \rightarrow \R$ as follows:
\begin{eqnarray*}
 \zeta(v,\theta) & = & \langle \bar{w};v \rangle_{\L^2(0,2;U)} + \bar{r} \theta.
\end{eqnarray*}
The first-order optimality condition reads
\begin{equation} \label{eq:astractKKT}
DJ(\bar{u},\bar{\tau}) - \bar{\lambda} \| (w,r) \| \zeta= 0.
\end{equation}
For all $\varepsilon \geq 0$, we define the cone $\mathcal{C}_\varepsilon$ as follows:
\begin{equation*}
\mathcal{C}_\varepsilon = \big\{ (v,\theta) \in \L^2(0,2;U) \times \R \,:\, |\zeta(v,\theta)| \leq \varepsilon \| (v,\theta) \| \big\}.
\end{equation*}
The cone $\mathcal{C}_0$ is called critical cone.
The following lemma is a metric regularity property.
In the prooofs of the following results, $C$ is used as a generic constant.

\begin{lemma} \label{lemmaRegM}
 There exist $\varepsilon>0$, a constant $\bar{C}>0$, and a twice continuously differentiable mapping $\varphi\colon (u,\tau) \in B_{\varepsilon}(\bar{u},\bar{\tau}) \rightarrow \L^2(0,2;U) \times (0,T)$ such that, for all $(u,\tau) \in B_{\varepsilon}(\bar{u},\bar{\tau})$,
 \begin{eqnarray}
 G \big(\varphi(u,\tau)\big)= 0 & \text{and} &
 \big\| \varphi(u,\tau) -(u,\tau) + \zeta(u-\bar{u},\tau-\bar{\tau})(\bar{w},\bar{r}) \big\| \leq \bar{C} \| (u-\bar{u},\tau-\bar{\tau}) \|^2. \label{eqRegM1}
 \end{eqnarray}
\end{lemma}

\begin{proof}
 Consider the following twice continuously differentiable mapping:
 \begin{eqnarray} \label{eqRegM0}
 \chi\colon (u,\tau,\pi) \in \L^2(0,2;U) \times (0,T) \times \R & \mapsto & G\big( (u,\tau) + \pi(\bar{w},\bar{r}) \big) \in \R.
 \end{eqnarray}
 We have: $\chi(\bar{u},\bar{\tau},0)= 0$, $\chi_\pi(\bar{u},\bar{\tau},0)= \langle w;\bar{w} \rangle_{\L^2(0,2;U)} + r\bar{r} = \| (w,r) \| \neq 0$. Therefore, by the implicit function theorem, the non-linear equation $\chi(u,\tau,\pi)= 0$ with unknown $\pi \in \R$ possesses a solution for $(u,\tau)$ close enough to $(\bar{u},\bar{\tau})$. More precisely, there exist $\varepsilon > 0$ and a twice continuously differentiable mapping $\Pi\colon B_{\varepsilon}(\bar{u},\bar{\tau}) \rightarrow 0$ such that $\Pi(\bar{u},\bar{\tau})= 0$ and such that
 \begin{equation} \label{eqRegM2}
 G\big( (u,\tau) + \Pi(u,\tau)(\bar{w},\bar{r}) \big)= 0, \quad \forall (u,\tau) \in B_{\varepsilon}(\bar{u},\bar{\tau}).
 \end{equation}
 Differentiating \eqref{eqRegM2} with respect to $(u,\tau)$ at $(\bar{u},\bar{\tau},0)$, we obtain
 \begin{equation*}
 D \Pi(\bar{u},\bar{\tau}).(v,\theta)= - \zeta(v,\theta), \quad \forall (v,\theta) \in \L^2(0,2;U) \times \R.
 \end{equation*}
 Since $\Pi$ is twice continuously differentiable, there exists a constant $C$ such that for all $(u,\tau) \in B_{\varepsilon}(\bar{u},\bar{\tau})$,
 \begin{eqnarray*}
  | \Pi(u,\tau)- \underbrace{\Pi(\bar{u},\bar{\tau})}_{=0} - D\Pi(\bar{u},\bar{\tau})(u-\bar{u},\tau-\bar{\tau}) |
  & \leq & C \| (u-\bar{u},\tau-\bar{\tau}) \|^2.
 \end{eqnarray*}
 Therefore, there exists a constant $C$ such that for all $(u,\tau) \in B_{\varepsilon}(\bar{u},\bar{\tau})$,
 \begin{eqnarray} \label{eqRegM3}
 \| \Pi(u,\tau)(\bar{w},\bar{r}) +\zeta(u-\bar{u},\tau-\bar{\tau}) \|
 & \leq & C \| (u-\bar{u},\tau-\bar{\tau}) \|^2.
 \end{eqnarray}
 We finally define $\varphi \colon B_{\varepsilon}(\bar{u},\bar{\tau}) \rightarrow \L^2(0,2;U) \times (0,T)$ by
 \begin{equation*}
 \varphi(u,\tau)= (u,\tau) + \Pi(u,\tau) (\bar{w},\bar{r}).
 \end{equation*}
 The two properties specified in \eqref{eqRegM0} follow directly from \eqref{eqRegM2} and \eqref{eqRegM3}.
\end{proof}

\begin{proposition}
 Let $(\bar{u},\bar{\tau})$ be locally optimal. Furthermore, assume that $G(\bar{u},\bar{\tau})= 0$.
 Then,
 \begin{equation*}
 D^2 L(\bar{u},\bar{\tau},\bar{\lambda}). \big( (v,\theta),(v,\theta) \big) \leq 0, \quad
 \forall (v,\theta) \in \mathcal{C}_0.
 \end{equation*}
\end{proposition}

\begin{proof}
 Let $(v,\theta) \in \mathcal{C}_0$. Let $(\varepsilon_k)_{k \in \mathbb{N}} \downarrow 0$. Set $(u_k,\tau_k)= (\bar{u},\bar{\tau}) + \varepsilon_k (v,\theta)$.
 For $k$ large enough, Lemma \ref{lemmaRegM} applies and therefore,
 \begin{equation} \label{eq:CN2add}
 \| \varphi (u_k,\tau_k) - (u_k,\tau_k) \| \leq C \varepsilon_k^2,
 \end{equation}
 since $\zeta(u_k-\bar{u},\tau_k-\bar{\tau})= \varepsilon_k \zeta(v,\theta)= 0$.
 We have $\varphi(u_k,\tau_k) \rightarrow (\bar{u},\bar{\tau})$, thus together with the feasibility of $\varphi(u_k,\tau_k)$, for $k$ large enough,
 \begin{equation*}
 0 \leq J \big( \varphi(u_k,\tau_k) \big) -J(\bar{u},\bar{\tau}).
 \end{equation*}
 Since $G(\bar{u},\bar{\tau})= 0$ and since $G(\varphi(u_k,\tau_k))= 0$ for all $k \in \mathbb{N}$,
 \begin{equation*}
 0 \leq J \big( \varphi(u_k,\tau_k) \big) -J(\bar{u},\bar{\tau})=
 L(\varphi(u_k,\tau_k),\bar{\lambda}) - L(\bar{u},\bar{\tau},\bar{\lambda}),
 \end{equation*}
 for $k$ large enough.
 The Lagrangian $L$ is twice continuously differentiable and $DL(\bar{u},\bar{\tau})= 0$, and so there exists $\mu_k \in [0,1]$ such that
 \begin{equation} \label{eqCN2}
 0 \leq \frac{1}{\varepsilon_k^2} \big( L(\varphi(u_k,\tau_k),\bar{\lambda}) - L(\bar{u},\bar{\tau},\bar{\lambda}) \big)=  D^2 L(\tilde{u}_k,\tilde{\tau}_k,\lambda). \left( \frac{\varphi(u_k,\tau_k)-(\bar{u},\bar{\tau})}{\varepsilon_k}, \frac{\varphi(u_k,\tau_k)-(\bar{u},\bar{\tau})}{\varepsilon_k} \right),
 \end{equation}
 where $(\tilde{u}_k,\tilde{\tau}_k)= \mu_k (\bar{u},\bar{\tau}) + (1-\mu_k) (u_k,\tau_k)$.
 Using \eqref{eq:CN2add}, we obtain that:
 \begin{equation*}
 \left\| \frac{\varphi(u_k,\tau_k)-(\bar{u},\bar{\tau})}{\varepsilon_k} - (v,\theta) \right\|
 = \frac{1}{\varepsilon_k} \| \varphi(u_k,\tau_k)-(u_k,\tau_k) \| \leq C \varepsilon_k
 \underset{k \to \infty}{\longrightarrow} 0.
 \end{equation*}
 Moreover, $(\tilde{u}_k,\tilde{\tau}_k) \underset{k \to \infty}{\longrightarrow} (\bar{u},\bar{\tau})$, therefore, we can pass to the limit in \eqref{eqCN2}. We obtain
 \begin{equation*}
 0 \leq D^2 L(\bar{u},\bar{\tau},\bar{\lambda}). (v,\theta)^2,
 \end{equation*}
 which concludes the proof.
\end{proof}

Consider now the following sufficient second-order optimality condition: There exists $\alpha > 0$ such that
\begin{eqnarray} \label{eqCS2assumption}
D^2L(\bar{u},\bar{\tau},\bar{\lambda}).(v,\theta)^2
\leq -\alpha \| (v,\theta) \|^2, & & \forall (v,\theta) \in \mathcal{C}_0.
\end{eqnarray}

\begin{lemma} \label{lemmaCritcone}
 Assume that the sufficient second-order optimality condition \eqref{eqCS2assumption} holds. Then, for all $0 < \beta < \alpha$, there exists $\varepsilon>0$ such that
 \begin{eqnarray*}
  D^2L(\bar{u},\bar{\tau},\bar{\lambda})(v,\theta)^2
  \leq -\beta \| (v,\theta) \|^2, & &
  \forall (v,\theta) \in \mathcal{C}_\varepsilon.
 \end{eqnarray*}
\end{lemma}

\begin{proof}
 To simplify the notation, we write $D^2L$ instead of $D^2L(\bar{u},\bar{\tau},\bar{\lambda})$.
 Let $\varepsilon > 0$ and let $(v,\theta) \in \mathcal{C}_\varepsilon$. Let us set $(v',\theta')= \zeta(v,\theta) (\bar{w},\bar{r})$.
 We have
 \begin{eqnarray*} %\label{eqLinRegM}
  (v,\theta) - (v',\theta') & \in & \mathcal{C}_0,
 \end{eqnarray*}
 since $\zeta(\bar{w},\bar{r})= 1$. Therefore,
 \begin{eqnarray}
 D^2L. (v-v',\theta-\theta')^2
 & \leq & -\alpha \| (v-v',\theta-\theta') \|^2 \notag \\
 & = & -\alpha \big( \| (v,\theta) \|^2 + \| (v',\theta') \|^2 - 2\langle v;v' \rangle_{\L^2(0,2;U)} - \theta \theta' \big) \notag \\
 & = & -\alpha \big( \| (v,\theta) \|^2 - \zeta(v,\theta)^2 \big) \notag \\
 & \leq & (-\alpha + \alpha \varepsilon^2) \| (v,\theta) \|^2. \label{eqFormQuad1}
 \end{eqnarray}
 We also have
 \begin{eqnarray}
 D^2 L. (v-v',\theta-\theta')^2
 & = & D^2L. (v,\theta)^2
 - 2 D^2L. \big( (v,\theta),(v',\theta') \big)
 + D^2 L. (v',\theta')^2 \notag \\
 & = & D^2 L. (v,\theta)^2
 -2 \zeta(v,\theta) D^2L. \big( (v,\theta),(\bar{w},\bar{r}) \big)
 + \zeta(v,\theta)^2 D^2L. (\bar{w},\bar{r})^2 \notag \\
 & \geq & D^2 L.(v,\theta)^2 - C(\varepsilon + \varepsilon^2) \| (v,\theta) \|^2.
 \label{eqFormQuad2}
 \end{eqnarray}
 Combining \eqref{eqFormQuad1} and \eqref{eqFormQuad2}, we obtain
 \begin{eqnarray*}
  D^2 L. \big( (v,\theta),(v,\theta) \big)
  & \leq & (-\alpha + C\varepsilon + C \varepsilon^2) \| (v,\theta) \|^2.
 \end{eqnarray*}
 For any $\beta \in (0,\alpha)$, there exists $\varepsilon>0$ small enough so that $-\alpha + C\varepsilon + C \varepsilon^2 \leq -\beta$. The lemma is proved.
\end{proof}

\begin{proposition}
 Assume that the sufficient second-order optimality condition \eqref{eqCS2assumption} holds and that $\bar{\lambda} > 0$.
 Then, for all $\beta \in (0,\alpha)$, there exists $\varepsilon>0$ such that for all $(u,\tau) \in B_{\varepsilon}(\bar{u},\bar{\tau})$, if $G(u,\tau) \leq 0$, then
 \begin{eqnarray*}
  J(u,\tau)-J(\bar{u},\bar{\tau}) & \leq &
  -\frac{1}{2} \beta  \| (u-\bar{u},\tau-\bar{\tau}) \|^2.
 \end{eqnarray*}
\end{proposition}

\begin{proof}
 We prove the result by contradiction. If such an $\varepsilon>0$ does not exist, then there exists a convergent sequence $(u_k,\tau_k)_{k \in \mathbb{N}}$ with limit $(\bar{u},\bar{\tau})$ such that for all $k \in \mathbb{N}$, $G(u_k,\tau_k) \leq 0$ and such that
 \begin{eqnarray} \label{eqCS22}
 J(u_k,\tau_k)-J(\bar{u},\bar{\tau}) & > &
 -\frac{1}{2} \beta  \| (u_k-\bar{u},\tau_k-\bar{\tau}) \|^2.
 \end{eqnarray}
 Inequality \eqref{eqCS22} implies that $(u_k,\tau_k) \neq (\bar{u},\bar{\tau})$. We set
 \begin{equation*}
 (v_k,\theta_k)= \frac{(u_k-\bar{u},\tau_k-\bar{\tau})}{\| (u_k-\bar{u},\tau_k-\bar{\tau}) \|}.
 \end{equation*}
 Since $J$ is twice continuously differentiable and since $(u_k,\tau_k)_{k \in \mathbb{N}}$ is a bounded sequence, there exists a constant $C > 0$ such that
 \begin{eqnarray}
 J(u_k,\tau_k)-J(\bar{u},\bar{\tau})
 & \leq & DJ(\bar{u},\bar{\tau})(u_k-\bar{u},\tau_k-\bar{\tau}) + C \| (u_k-\bar{u},\tau_k-\bar{\tau}) \|^2 \notag \\
 & = & \bar{\lambda} \| (w,r) \| \zeta(v_k,\theta_k) \| (u_k-\bar{u},\tau_k-\bar{\tau} ) \| + C \| (u_k-\bar{u},\tau_k-\bar{\tau}) \|^2, \label{eqCS2a}
 \end{eqnarray}
 where we used \eqref{eq:astractKKT}.
 Since $\bar{\lambda} >0$, we obtain by combining \eqref{eqCS22} and \eqref{eqCS2a} that
 \begin{equation*}
 \zeta(v_k,\theta_k) \geq -C \| (u_k-\bar{u},\tau_k-\bar{\tau}) \| \underset{k \to \infty}{\longrightarrow} 0.
 \end{equation*}
 Thus, $\liminf_{k \to \infty} \zeta(v_k,\theta_k) \geq 0$.
 Since $(u_k,\tau_k)$ is bounded and feasible and since $G$ is twice continuously differentiable, there exists a constant $C>0$ such that for all $k \in \mathbb{N}$,
 \begin{equation*}
 0 \geq G(u_k,\tau_k)-G(\bar{u},\bar{\tau})
 = \| (w,r) \| \zeta(v_k,\theta_k) \| (u_k-\bar{u},\tau_k-\bar{\tau}) \| - C \| (u_k-\bar{u},\tau_k-\bar{\tau}) \|^2.
 \end{equation*}
 Therefore,
 \begin{equation*}
 \zeta(v_k,\theta_k) \leq C \| (u_k-\bar{u},\tau_k-\bar{\tau}) \| \underset{k \to \infty}{\longrightarrow} 0.
 \end{equation*}
 Thus, $\limsup_{k \to \infty} \zeta(v_k,\theta_k) \leq 0$ and finally, $\zeta(v_k,\theta_k) \underset{k \to \infty}{\longrightarrow} 0$.
 
 We are ready to obtain a contradiction to \eqref{eqCS22}. We have $G(\bar{u},\bar{\tau})= 0$. Moreover, for all $k \in \mathbb{N}$, $G(u_k,\tau_k) \leq 0$. Thus
 \begin{equation} \label{eqCS2A}
 J(u_k,\tau_k)-J(\bar{u},\bar{\tau})
 \leq L(u_k,\tau_k,\bar{\lambda}) - L(\bar{u},\bar{\tau},\bar{\lambda}).
 \end{equation}
 Since $L$ is twice continuously differentiable and since $DL(\bar{u},\bar{\tau},\bar{\lambda})= 0$, there exists $\mu_k \in [0,1]$ such that
 \begin{eqnarray}
 L(u_k,\tau_k,\bar{\lambda}) - L(\bar{u},\bar{\tau},\bar{\lambda})
 & = & \frac{1}{2} \| (u_k-\bar{u},\tau_k-\bar{\tau}) \|^2
 D^2 L(\tilde{u}_k,\tilde{\tau}_k,\bar{\lambda}). (v_k,\theta_k)^2,
 \label{eqCS2B}
 \end{eqnarray}
 where $(\tilde{u}_k,\tilde{\tau}_k)= (1-\mu_k) (\bar{u},\bar{\tau}) + \mu_k (u_k,\tau_k)$.
 Observe that $(\tilde{u}_k,\tilde{\tau}_k) \underset{k \to \infty}{\longrightarrow} (\bar{u},\bar{\tau})$.
 We introduce now a number $\eta \in (\beta, \alpha)$.
 For $k$ large enough, $\| D^2 L(\tilde{u}_k,\tilde{\tau}_k,\bar{\lambda}) - D^2L(\bar{u},\bar{\tau},\bar{\lambda}) \| \leq \eta-\beta$. Thus, combining \eqref{eqCS2A} and \eqref{eqCS2B}, we obtain
 \begin{eqnarray} \label{eqCS2C}
 J(u_k,\tau_k)-J(\bar{u},\bar{\tau})
 & \leq & \frac{1}{2} \|(u_k-\bar{u},\tau_k-\bar{\tau})\|^2 \left( D^2L(\bar{u},\bar{\tau},\bar{\lambda})(v_k,\theta_k)^2 + \eta - \beta \right).
 \end{eqnarray}
 By Lemma \ref{lemmaCritcone}, there exists $\varepsilon > 0$ such that
 \begin{eqnarray*}
  D^2 L(\bar{u},\bar{\tau},\bar{\lambda}).(v,\theta)^2 \leq -\eta \| (v,\theta) \|^2, & &
  \forall (v,\theta) \in \mathcal{C}_{\varepsilon}.
 \end{eqnarray*}
 For $k$ large enough, $\zeta(v_k,\theta_k) \in \mathcal{C}_{\varepsilon}$ since $\zeta(v_k,\theta_k) \rightarrow 0$.
 Therefore, by \eqref{eqCS2C},
 \begin{equation*}
 J(u_k,\tau_k)-J(\bar{u},\bar{\tau})
 \leq -\frac{1}{2} \beta \| (u_k-\bar{u},\tau_k-\bar{\tau}) \|^2,
 \end{equation*}
 for $k$ large enough, which contradicts \eqref{eqCS22}.
\end{proof}

\color{black}

\section{Numerical realization} \label{sec4}

\subsection{Method} \label{secpractice}
\textcolor{black}{ The numerical realization is based  on a optimize then discretize approach using  Newton's method for the reduced
 formulation. For this purpose we need the first and second order  sensitivity information  which was obtained in Propositions \ref{propJ1} and \ref{propJ2},  as well as the optimality conditions of Proposition~\ref{Prop3}.} For ease of computations we did not realize numerically the norm constraint, and thus the expressions of the derivatives in Propositions~\ref{propJ1} and~\ref{propJ2} hold for $J$, as well as the optimality conditions of Proposition~\ref{Prop3} (with $\lambda = 0$).
In order to reach  the region  of attraction  for  Newton's
method, first  some  Barzilai-Borwein gradient steps are performed
(see~\cite{BB} for instance), \textcolor{black}{and these steps are initialized with the Armijo rule}. In these gradient steps, only two inner
products are computed for each step. Next, when the norm of the gradient
is small enough, we switch to the full-step Newton algorithm, which is
faster, but demands one linear system solve at
each iteration. Switching to the Newton step is monitored by the norm
of the gradient given by
\begin{eqnarray*}
 ||| (J_u,J_\tau) ||| : = \| (J_u,J_\tau) \|_{\L^2(0,2;\R^m) \times
  \R}^2 & = & \int_0^2 |J_u|^2_{\R^m}(s)\d s + |J_\tau|_{\R}^2.
\end{eqnarray*}
The integral above is approximated by the trapezoidal rule. The
algorithm then performed is the following:
%\textcolor{black}{Start with Newton, then stopping criterion.}
%The norm of the gradient of $J$, that we want to vanish is given by
%\begin{eqnarray*}
% ||| (J_u,J_\tau) ||| : = \| (J_u,J_\tau) \|_{\L^2(0,2;\R^m) \times \R}^2 & = & \int_0^2 |J_u|^2_{\R^m}(s)\d s + |J_\tau|_{\R}^2.
%\end{eqnarray*}
%The integral above -- which appears also below -- is approximated by the trapezoidal rule. The algorithm then performed is the following:

\begin{algorithm}[htpb] %\tag{Algorithm}
 \begin{description}
  \item[Initialization:] $u = 0$, $\tau_0 = T/2$, for $1\leq i\leq N$, $s(i) = 2i/N$.
  \item[\textcolor{black}{Initial gradient:}] \textcolor{black}{Compute $(J_u,J_\tau)$ corresponding to the initial data above.}
  \item[Steps for vanishing the gradient:] \hfill
  \begin{description}
   \item[Gradient steps:] \textcolor{black}{Compute another $(J_u,J_\tau)$ with the Armijo rule.}
   \item[Barzilai-Borwein steps:] While $||| (J_u,J_\tau) ||| > 10^{-4}$, \textcolor{black}{perform this gradient method.}
   \item[Newton steps:] While $||| (J_u,J_\tau) ||| > 10^{-12}$, do:\\
   $\bullet$ Compute $(\delta u, \delta \tau)$ by solving system~\eqref{sysNewton}.\\
   $\bullet$ Update the unknowns: $u_{k+1} = u_k + \delta u$, $\tau_{k+1} = \tau_k + \delta \tau$.\\
   $\bullet$ Update the gradient $(J_u,J_\tau)$.
  \end{description}
  \item[Post-processing:] for $1\leq i\leq N$, $t(i) := \pi(s(i),\tau)$.
 \end{description}
 \caption{Solving the first-order optimality conditions}\label{algo}
\end{algorithm}
\FloatBarrier
%In order to catch the attraction area for the Newton's method, we first perform some Barzilai-Borwein gradient steps (see~\cite{BB} for instance). In these gradient steps, only two inner products are computed for each step. Next, when the norm of the gradient is small enough, we switch to the full-step Newton's algorithm, which is faster (quadratic convergence), but demands one linear system solve at each iteration.
%In order to catch the attraction area for the Newton's method, we first perform some Barzilai-Borwein gradient steps (see~\cite{BB} for instance). In these gradient steps, only two inner products are computed for each step. Next, when the norm of the gradient is small enough, we switch to the full-step Newton's algorithm, which is faster (quadratic convergence), but demands one linear system solve at each iteration.

Recall that $J$ is defined in~\eqref{pbmain2}. The derivatives of $J$ are provided by Proposition~\eqref{propJ1} and Proposition~\eqref{propJ2}. For solving one Newton step, we use the Gmres algorithm~\cite{GMRES}, calling only the evaluation of the mapping\textcolor{black}{
 \begin{eqnarray*}
  \left(\begin{matrix} v \\ \theta \end{matrix}  \right) & \mapsto &
  \left(\begin{matrix}J_{uu} & J_{u\tau} \\ J_{\tau u} & J_{\tau\tau} \end{matrix}  \right) .\left(\begin{matrix} v \\ \theta \end{matrix}  \right),
 \end{eqnarray*}
 in order to solve the system
 \begin{eqnarray}
 \left(\begin{matrix}J_{uu} & J_{u\tau} \\ J_{\tau u} & J_{\tau\tau} \end{matrix}  \right) .\left(\begin{matrix} \delta u \\ \delta \tau \end{matrix}  \right)
 & = & - \left(\begin{matrix}  J_u \\  J_{\tau} \end{matrix}  \right).
 \label{sysNewton}
 \end{eqnarray}
 This leads in particular, for $(u,\tau)$ and $y = S(u,\tau)$ given, to} the evaluation of the following quantities:
\begin{eqnarray*}
 J_{uu}.(\delta u, \cdot) & = & \dot{\pi}(\cdot,\tau) \Big( S_u^\ast H_{yy} S_u.\delta u + H_{uy}S_u.\delta u + S_u^\ast H_{yu}.\delta u + H_{uu}.\delta u \Big) \\
 & & + S_u^\ast D^2\phi_1(y(1))S_u.\delta u + S_u^\ast D^2\phi_2(y(2))S_u.\delta u , \\
 J_{\tau u} & = & \dot{\pi}(\cdot,\tau) \Big(  S_u^\ast H_y + H_u + H_{uy}S_\tau + S^\ast_u H_{yy}S_\tau \Big) +  S_u^\ast D^2\phi_1(y(1))S_\tau +  S_u^\ast D^2\phi_2(y(2))S_\tau.
\end{eqnarray*}
Note, in particular, that we do not need any evaluation of the adjoint operator $S_\tau^{\ast}$.

\subsection{Illustrations}
In all the tests below, for the cost function we choose
\begin{eqnarray*}
 \ell(y,u)  =  -\frac{\alpha}{2} |u|_{\R^m}^2,
\end{eqnarray*}
for a cost parameter $\alpha >0$. The examples dealing with ordinary differential systems, considered below, are inspired by~\cite{Trelat}. \textcolor{black}{The time discretizations, for solving the state equations as well as for solving the adjoint states, are made with the Crank-Nicolson scheme. This is a second-order scheme, which is important for getting a good accuracy for the time evolution, in particular for coupled systems whose the dynamics is complex.}

\subsubsection{The Lotka-Volterra prey-predator system}
Consider the following differential system:
\begin{eqnarray*}
 \left( \begin{matrix} \dot{y}_1 \\ \dot{y}_2 \end{matrix} \right) & = &
 \left( \begin{matrix} (y_1(a-by_2) + \text{\textcolor{black}{$u_1$}}y_1)(1-c_1y_1) \\ (y_2(qy_1-r)+  \text{\textcolor{black}{$u_2$}}y_2)(1-c_2y_2) \end{matrix} \right).
\end{eqnarray*}
In this system the variables $y_1$ and $y_2$ represent the densities of population of preys and predators, respectively. The multiplicative terms of type $(1-c_iy_i)$ are considered in order to limit the values of these densities to $1/c_i$. We assume that we can control the birth rates and death rates of both species, through the bilinear control made of $u_1$ and $u_2$. The numerical method given in section~\ref{secpractice} is applied to this system, with
\begin{eqnarray*}
 T = 30.0, \quad y_0 = (1.0,2.0)^T, \quad \alpha = 10.0, \quad a=0.3, \quad b=0.1, \quad r=0.2, \quad q=0.1, \quad c_1 =c_2 = 0.05.
\end{eqnarray*}
We do not consider any terminal cost, namely $\phi_2 \equiv 0$, and the functional we maximize at some time $\tau$ is $\phi_1(y) = y_2$. It means that we want to maximize the density of population of predators. The tests presented in Figures~\ref{figLV} and~\ref{figULV} are obtained with a Crank-Nicolson time-discretization, with $N = 3000$ time steps. %\textcolor{black}{This corresponds to a second-order scheme, which is important for getting a good accuracy for the time evolution, specially in our context where we optimize a time parameter.}\\

\begin{minipage}{\linewidth}
\centering
\begin{minipage}{0.45\linewidth}
\includegraphics[trim = 0cm 0cm 0cm 1cm, clip, scale=0.25]{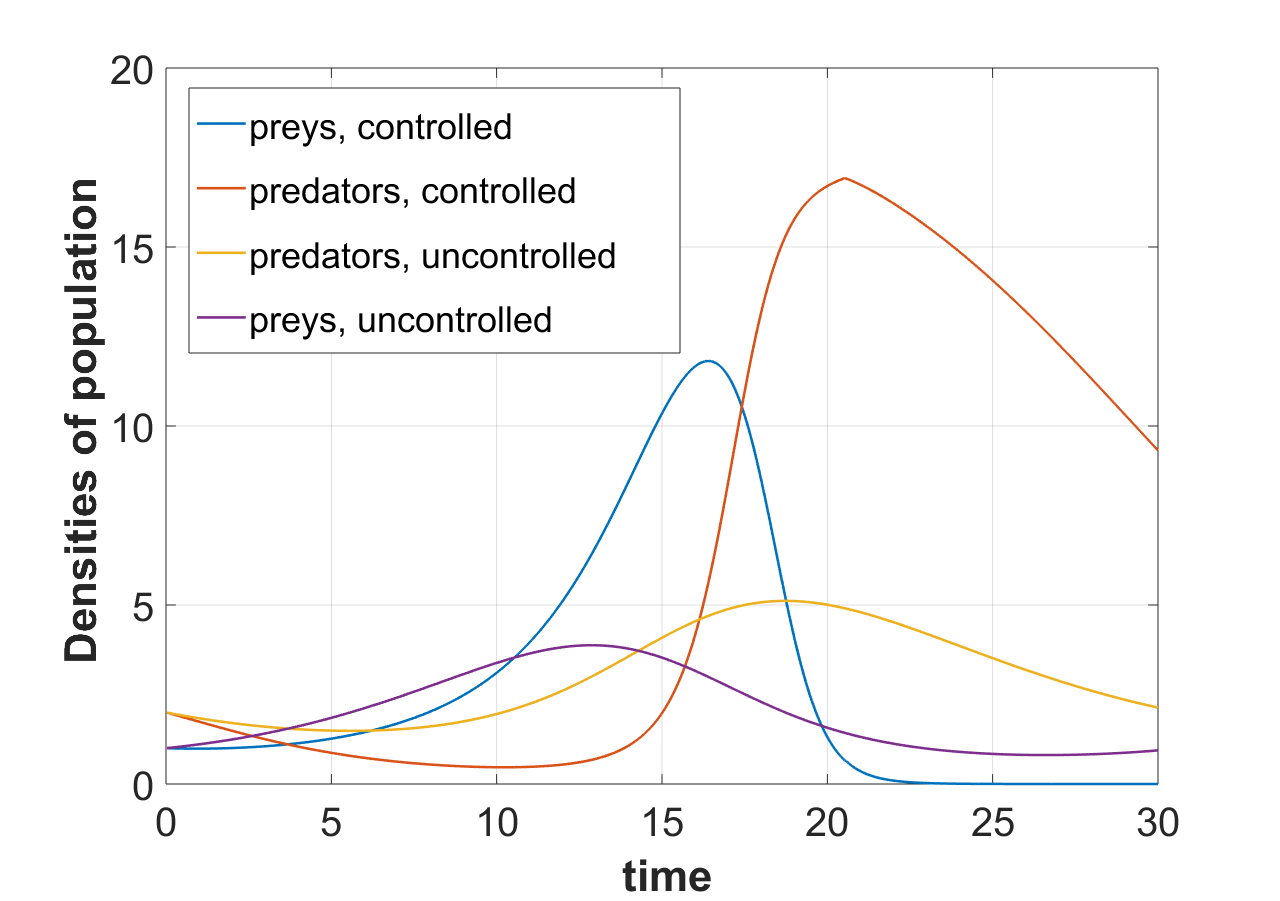}
\captionof{figure}{Comparison of the evolutions of states for the Lotka-Volterra system, with and without control.\label{figLV}}
\end{minipage}
\hspace{0.05\linewidth}
\begin{minipage}{0.45\linewidth}
\includegraphics[trim = 0cm 0cm 0cm 1cm, clip, scale=0.25]{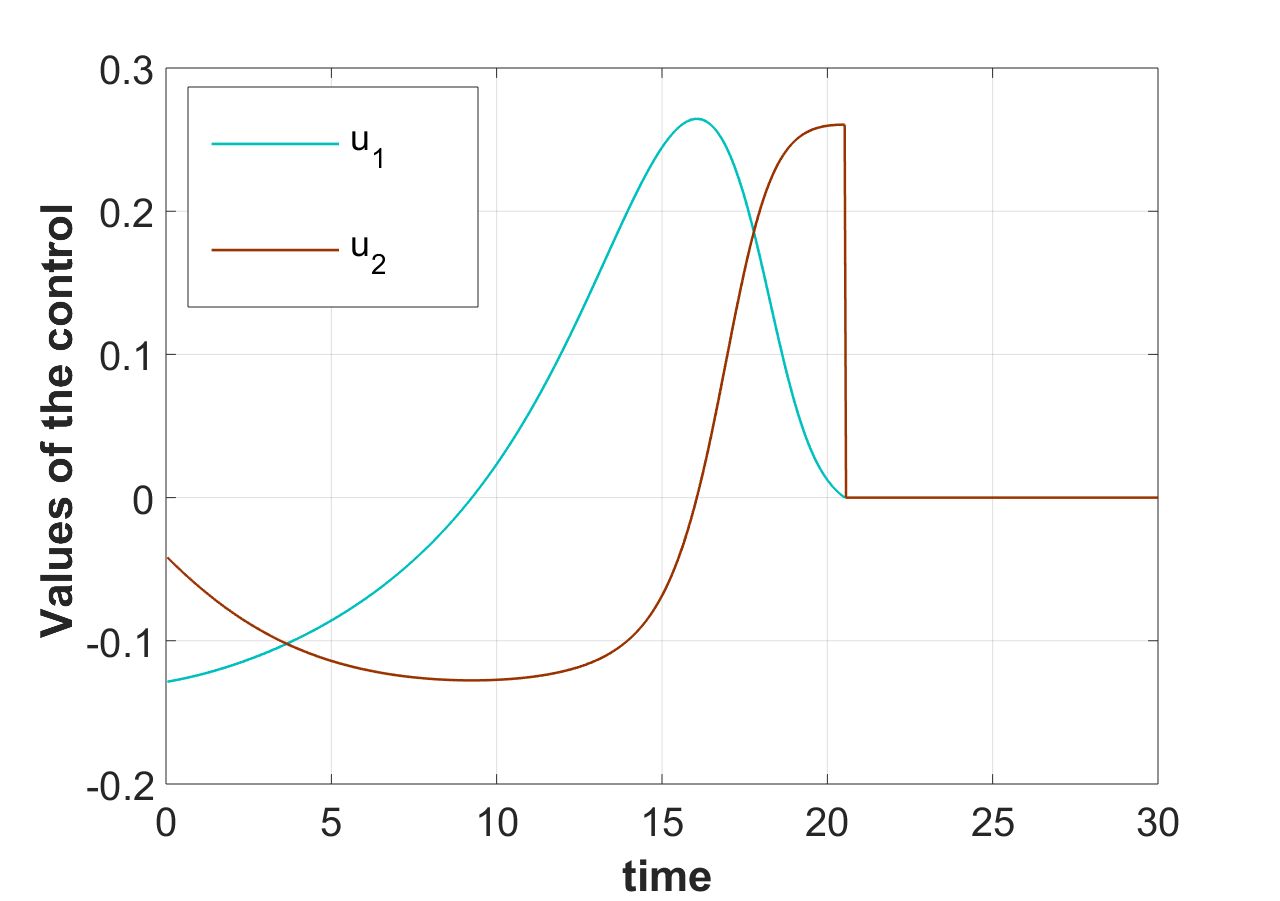}
\captionof{figure}{Values of the bilinear control through the time. \textcolor{white}{blablablablablablablablablablallllllllllllllllll}\label{figULV}}
\end{minipage}
\end{minipage}\\
\hfill \\
\FloatBarrier
At the beginning, the action of the control leads to diminution of the populations of both preys and predators, then leads to introduction of preys in order to feed the predators, and then favor their reproduction. The maximum of predators is reached at $\tau \approx 20.57$. As expected, the time-derivative of the second-component of the state has a jump. Note that in this example the first-order optimality condition gives in particular the equalities
\begin{eqnarray*}
 \alpha u_1 = y_1(1-c_1y_1)p_1, & & \alpha u_2 = y_2(1-c_2y_2)p_2.
\end{eqnarray*}
Without any terminal cost, the control is indeed null for $t > \tau$, as expected in system~\eqref{sysadj} defining the adjoint-state $(p_1,p_2)^T$.

\paragraph{With terminal cost.}
In order to avoid the extinction of population of the preys after having maximized the population of predators, we consider the terminal cost functional
\begin{eqnarray*}
 \phi_2(y) & = & -\beta \log \left( \left|\frac{y_1}{y_{\mathrm{des}}}\right| \right)^2,
\end{eqnarray*}
where $\beta > 0$ is a coefficient chosen large enough and $y_{\mathrm{des}}$ is the desired value for the density of preys at time $t=T$. The idea is to penalize the extinction of this population (the case $y_1(T) = 0$ is forbidden), and to force to reach the desired value by choosing $\beta$ large enough. With the same coefficients chosen as previously, the same time-discretization, and with $\beta = 25.0$ and $y_{\mathrm{des}} = 1.0$, we obtain the results presented in Figures~\ref{figterm1} and~\ref{figterm2}.\\

\begin{minipage}{\linewidth}
 \centering
 \begin{minipage}{0.45\linewidth}
   \includegraphics[trim = 0cm 0cm 0cm 1cm, clip, scale=0.25]{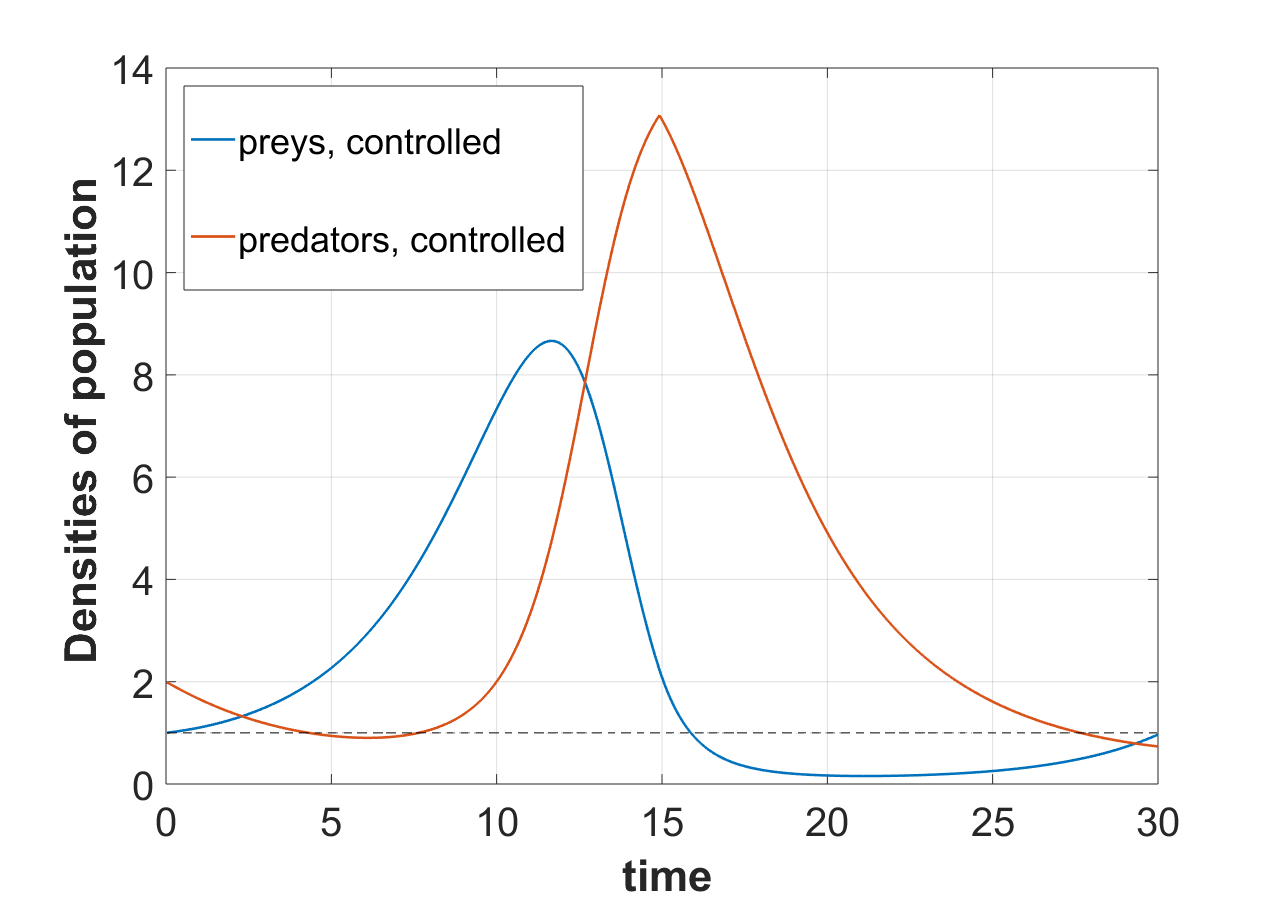}
   \captionof{figure}{Evolutions of the state for the Lotka-Volterra system, with control and terminal cost.\label{figterm1}}
 \end{minipage}
 \hspace{0.05\linewidth}
 \begin{minipage}{0.45\linewidth}
   \includegraphics[trim = 0cm 0cm 0cm 1cm, clip, scale=0.25]{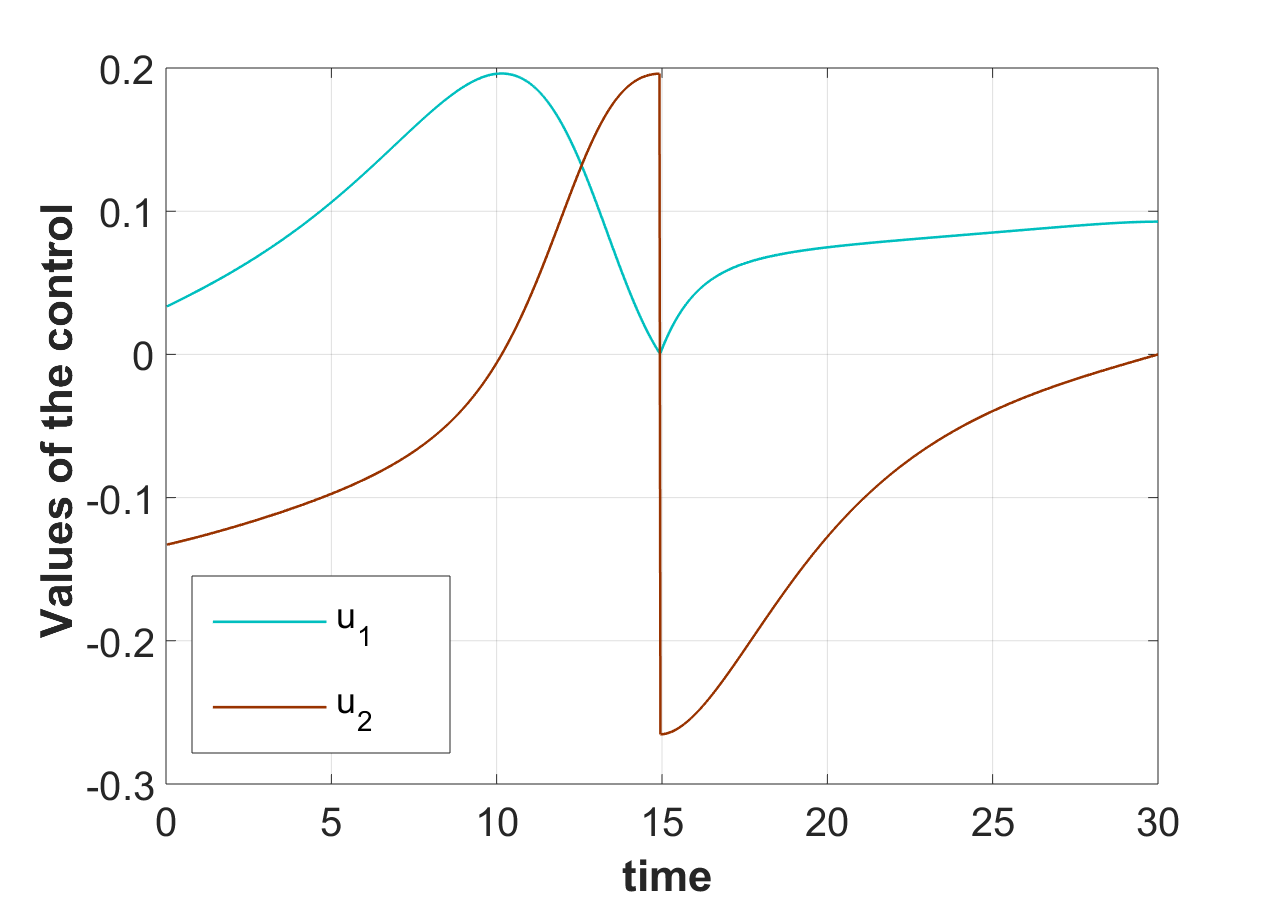}
   \captionof{figure}{Values of the bilinear control through the time, with terminal cost. \textcolor{white}{}\label{figterm2}}
 \end{minipage}
\end{minipage}\\
\hfill \\
\FloatBarrier
The optimal time for the maximum of predators is $\tau \approx 14.87$. As expected, the maximum is smaller than the one reached without terminal cost, because the control has to be activated for $t >\tau$ in order to take into account the functional $\phi_2$.

\subsubsection{The simple damped pendulum}
Consider the model of a simple pendulum, as described in Figure~\ref{figPendulum} below.
\begin{center}
 \scalebox{0.4}{
  \begin{picture}(0,0)%
  \centering
  \includegraphics[trim = 0cm 0cm 0cm 0cm, clip, scale=1.00]{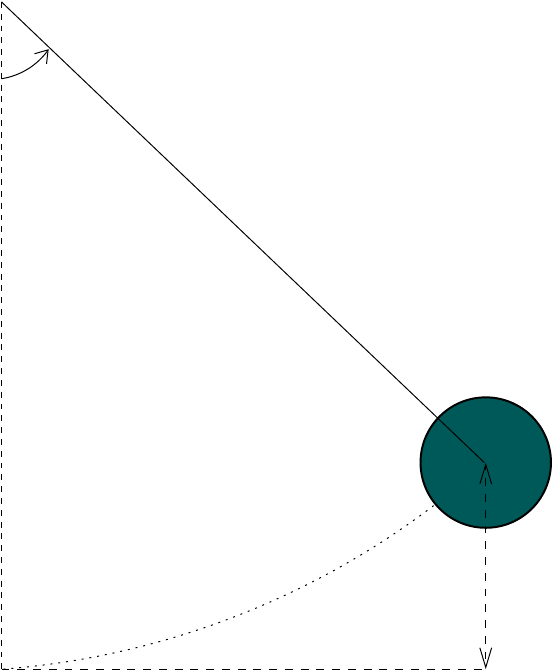}   
  %\caption{Simple pendulum.\label{figPendulum}}
  \end{picture}%
  \setlength{\unitlength}{4144sp}%
  \begin{picture}(4214,5109)(5884,-8083)
  \put(9721,-7396){\makebox(0,0)[lb]{\smash{{\SetFigFont{29}{34.8}{\rmdefault}{\mddefault}{\updefault}{\color[rgb]{0,0,0}$L\sin \theta$}%
     }}}}
     \put(6086,-3906){\makebox(0,0)[lb]{\smash{{\SetFigFont{29}{34.8}{\rmdefault}{\mddefault}{\updefault}{\color[rgb]{0,0,0}$\theta$}%
        }}}}
\end{picture}%
}
\end{center}
\FloatBarrier
\vspace{-20pt}
\begin{figure}[!h]
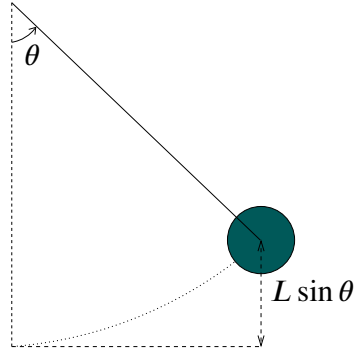

\centering  
\caption{Simple pendulum.\label{figPendulum}}
\end{figure}
\FloatBarrier
\vspace{-25pt}
The weight is assumed to satisfy the Newton's law, and so in particular the angle $\theta$ has to to satisfy the equation $\ddot{\theta} +\lambda\dot{\theta} +\mu\sin \theta = 0$, where $\lambda >0$ is a damping term, and $\mu$ is a coefficient depending on the gravity field and the length of the taut rope. the corresponding differential system is
\begin{eqnarray*}
       \left( \begin{matrix} \dot{y}_1 \\ \dot{y}_2 \end{matrix} \right) & = &
       \left( \begin{matrix} y_2 \\ -\lambda y_2 -\mu \sin y_1  \end{matrix} \right) + \left(\begin{matrix} 0 \\ u \end{matrix}\right).
\end{eqnarray*}
where $y = (y_1,y_2)^T := (\theta,\dot{\theta})^T$, and where the control $u$ represents some additional horizontal force in the Newton's law. The numerical method developed in section~\ref{secpractice} is performed, with
\begin{eqnarray*}
T = 25.0, \quad y_0 = (-1.0,0.0)^T, \quad \alpha = 10.0, \quad \lambda =0.03, \quad \mu =1.0.
\end{eqnarray*}
Here again, we do not consider any terminal cost, namely $\phi_2 \equiv 0$. The functional we maximize at some time $\tau$ is $\phi_1(y) = y_1$, namely the angle and thus the height of the weight. The results presented in Figures~\ref{figP} and~\ref{figUP} are obtained with a Crank-Nicolson time-discretization, with $N = 2500$ time steps.\\
\begin{minipage}{\linewidth}
\centering
\begin{minipage}{0.45\linewidth}
\includegraphics[trim = 0cm 0cm 0cm 1cm, clip, scale=0.25]{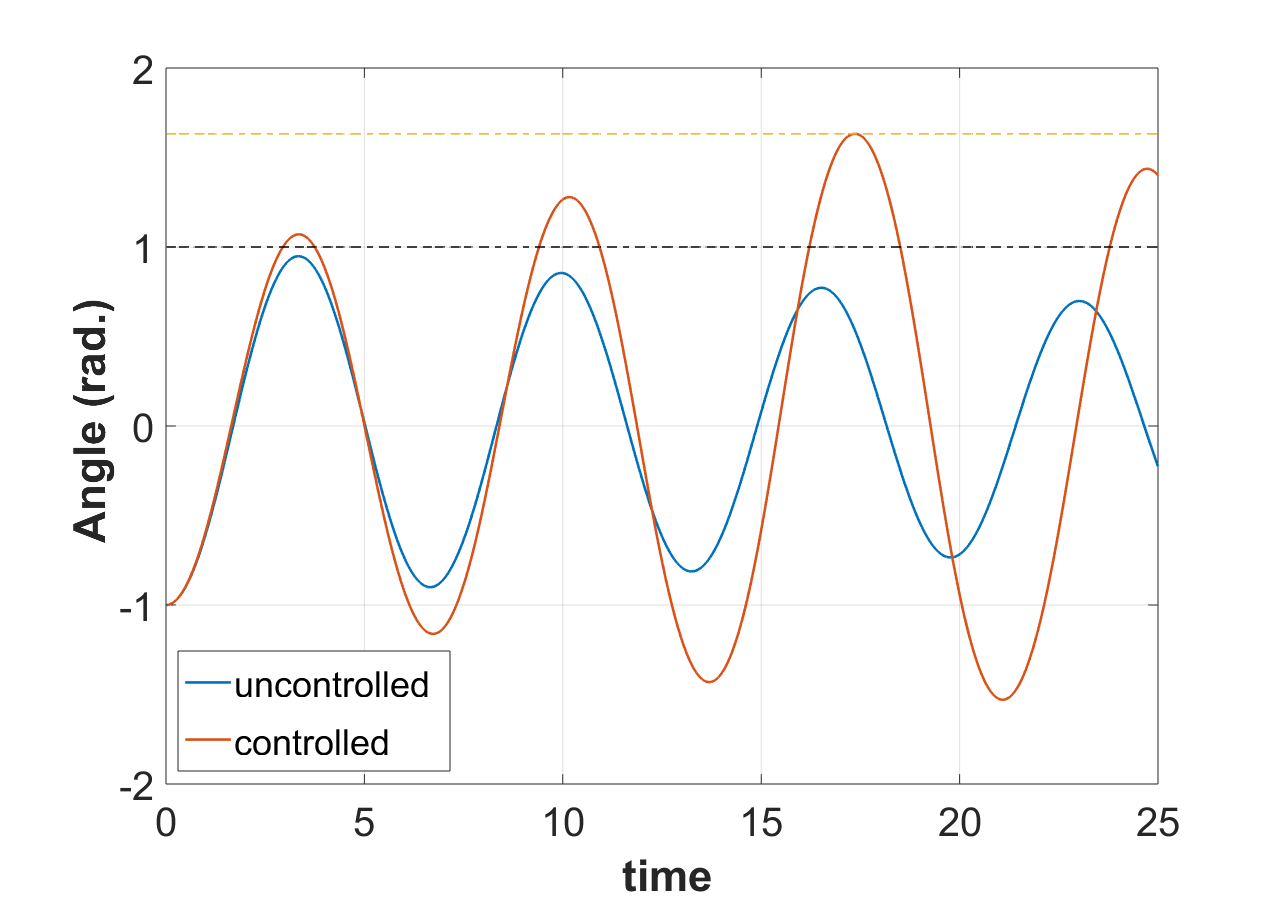}
\captionof{figure}{Comparison of the evolutions of angles for the damped simple pendulum, with and without control.\label{figP}}
       \end{minipage}
       \hspace{0.05\linewidth}
       \begin{minipage}{0.45\linewidth}
\includegraphics[trim = 0cm 0cm 0cm 1cm, clip, scale=0.25]{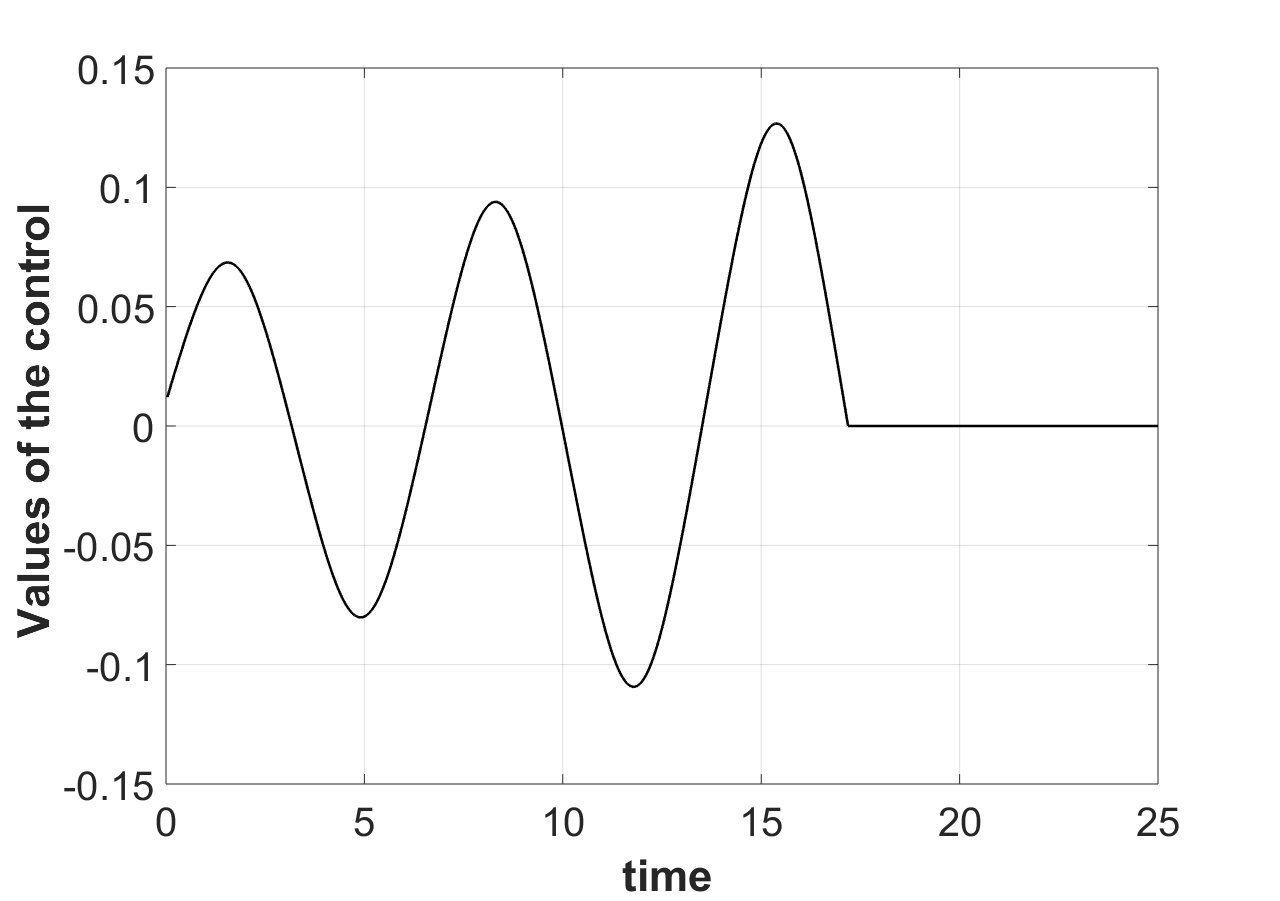}
\captionof{figure}{Values of the control through the time.\textcolor{white}{blablbalalalalalalal}
\hfill  \label{figUP}}
\end{minipage}
\end{minipage}\\
\hfill \\
\FloatBarrier
The maximum is reached at $\tau \approx 17.22$. As we can see, the control can be kept activated on several pseudo-periods, and it can compensate the damping effect. But numerically, the difficulty lies in considering an appropriate initialization for $\tau$ (which is not $T/2$ in that case). Indeed, our numerical approach enables us only to get a critical point, and thus without a good initialization, one might find  some local extremum, instead of the desired maximum. \textcolor{black}{A solution consists in modifying this initialization {\it a posteriori}, if the critical point first found is undesired. An {\it a priori} knowledge of a time interval in which $\tau$ would lie may be required.}
      
      \subsubsection{A partial differential equation} \label{secnumPDE}
      Consider the following Burgers-type system
      \begin{eqnarray*}
       \left\{ \begin{array} {ll}
        \dot{y}  = \nu y_{xx} - \beta yy_x + \mathbf{1}_{\omega} u, & (x,t) \in (0,1) \times (0,T), \\
        y(0,t) = y(1,t) = 0, & t \in (0,T), \\
        y(x,0) = 10.(1 - e^{-(1-x)})(e^{-(1-x)} - e^{-1}), & x \in (0,1),
        %\sin (\pi x), & x \in (0,1),
       \end{array} \right.
      \end{eqnarray*}
      where $\nu$ and $\beta$ are positive constant. Here the unknown $y$ is considered in the space $W(0,T; \H^1_0(0,1))$ (defined in section~\ref{secset}), corresponding to the Gelfand triplet $\H^1_0(0,1) \hookrightarrow \L^2(0,1) \hookrightarrow \H^{-1}(0,1)$. Recall that $\H^1(0,1) \hookrightarrow \mathcal{C}([0,1])$. The question of well-posedness for such a system is addressed in~\cite{Volkwein}. The control $u$ is considered in $\L^2(\omega)$, with $\omega = [0.00 ; 0.25]$, and the cost term is given by $\ell(y,u) = -\frac{\alpha}{2} \|u\|^2_{\L^2(\omega)}$. We want to maximize, at some optimal time $\tau \in (0,T)$, the following quantity
      \begin{eqnarray*}
       \phi_1(y(\cdot,\tau)) & = & \frac{1}{2}\int_D |y(x,\tau)|^2\d x
      \end{eqnarray*}
      with $D = [0.25 ; 0.30]$, and without considering any terminal cost: $\phi_2 \equiv 0$. \textcolor{black}{Recall that the Burgers equation is the one-dimensional version of the Navier-Stokes equations. The state $y$ plays the role of a velocity, and thus the quantity we maximize in $\phi_1$ corresponds to a kinetic energy.} The space discretization is done with finite P1-elements, with $n = 101$ degrees of freedom for the state variable, and hence $m=26$ unknowns for the control variable. The time discretization is done with a Crank-Nicolson scheme, with $N = 1000$ time steps. We consider the following parameters
      \begin{eqnarray*}
       T = 10.0, \quad \alpha = 2.10^{-9}, \quad \nu = 2.10^{-4}, \quad \beta = 0.05.
      \end{eqnarray*}
      The evolutions of the state and the control through the time are represented in Figure~\ref{figBurgers} below.\\
\begin{minipage}{\linewidth}
\centering
\begin{minipage}{0.2175\linewidth}
\includegraphics[trim = 0cm 0cm 0cm 1cm, clip, scale=0.15]{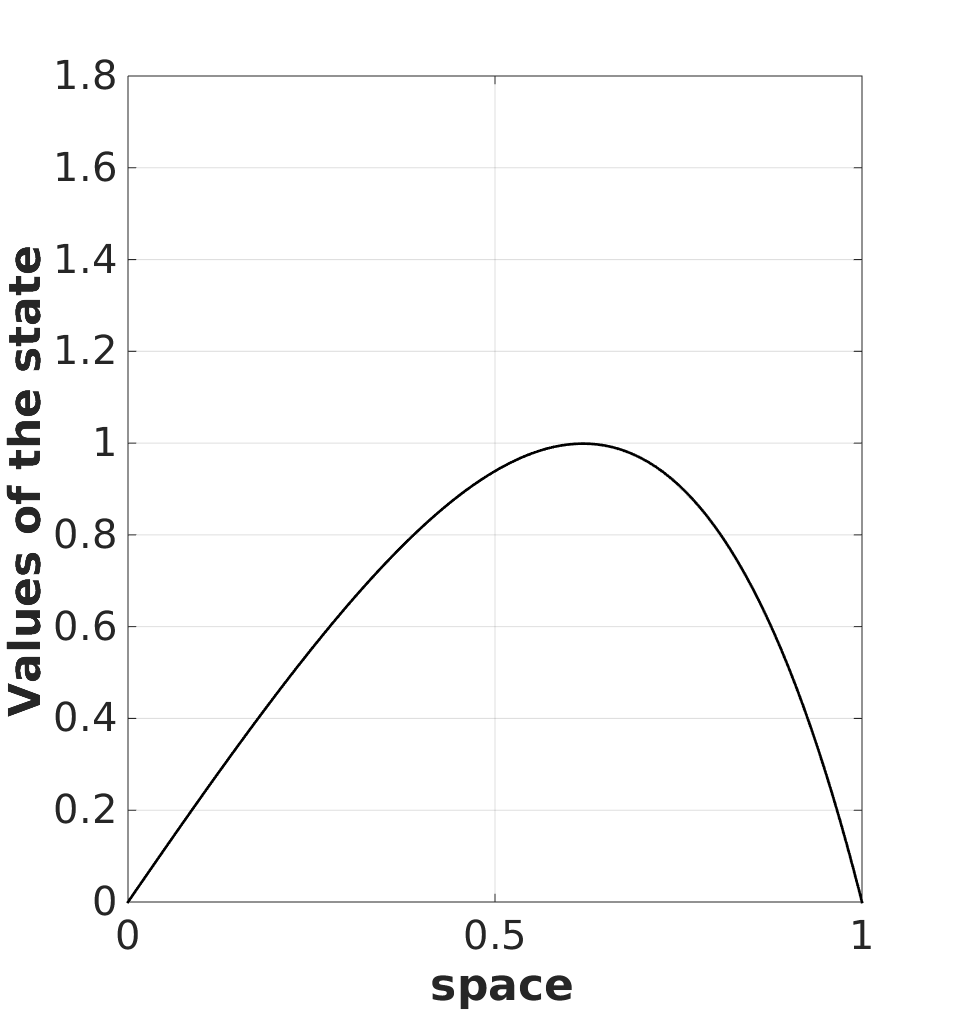} \\
\includegraphics[trim = 0cm 0cm 0cm 0cm, clip, scale=0.15]{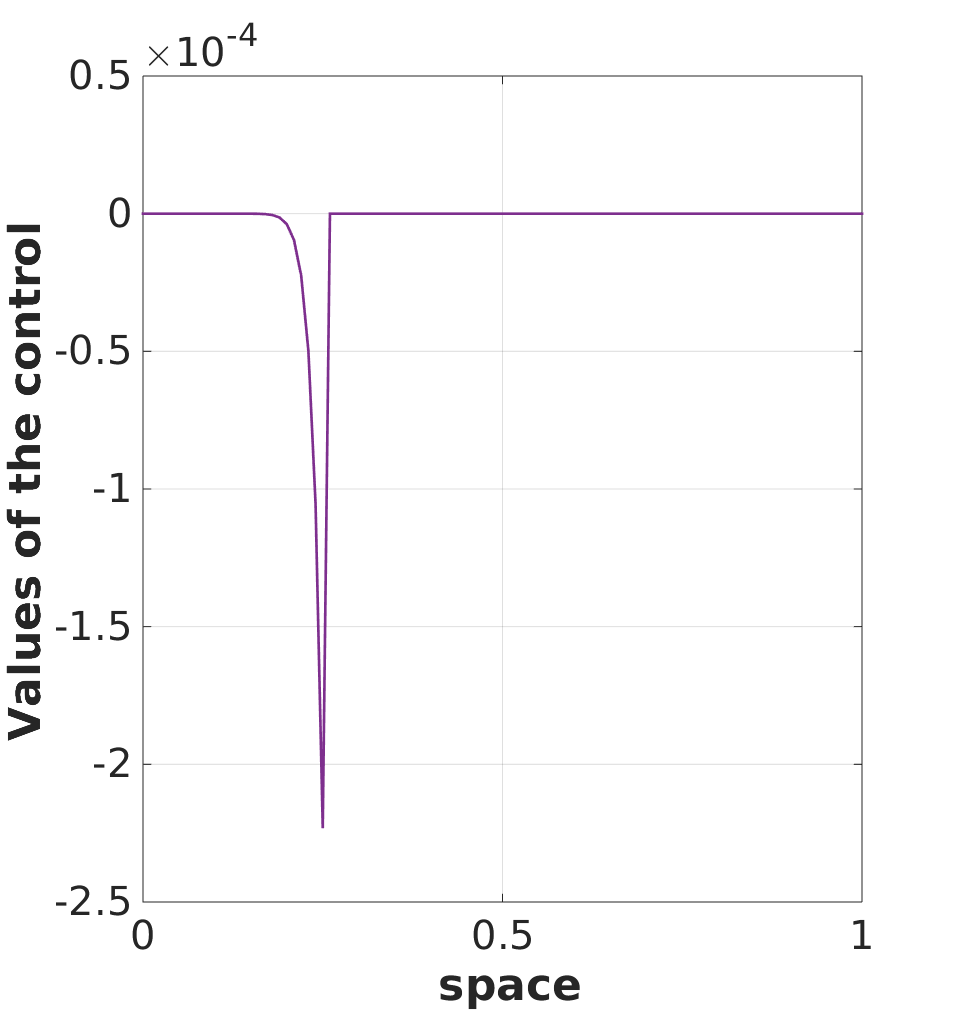}
\begin{center} $ t = 0.0777 $ \end{center}
\end{minipage}
\hspace{0.03\linewidth}
\begin{minipage}{0.2175\linewidth}
\includegraphics[trim = 0cm 0cm 0cm 1cm, clip, scale=0.15]{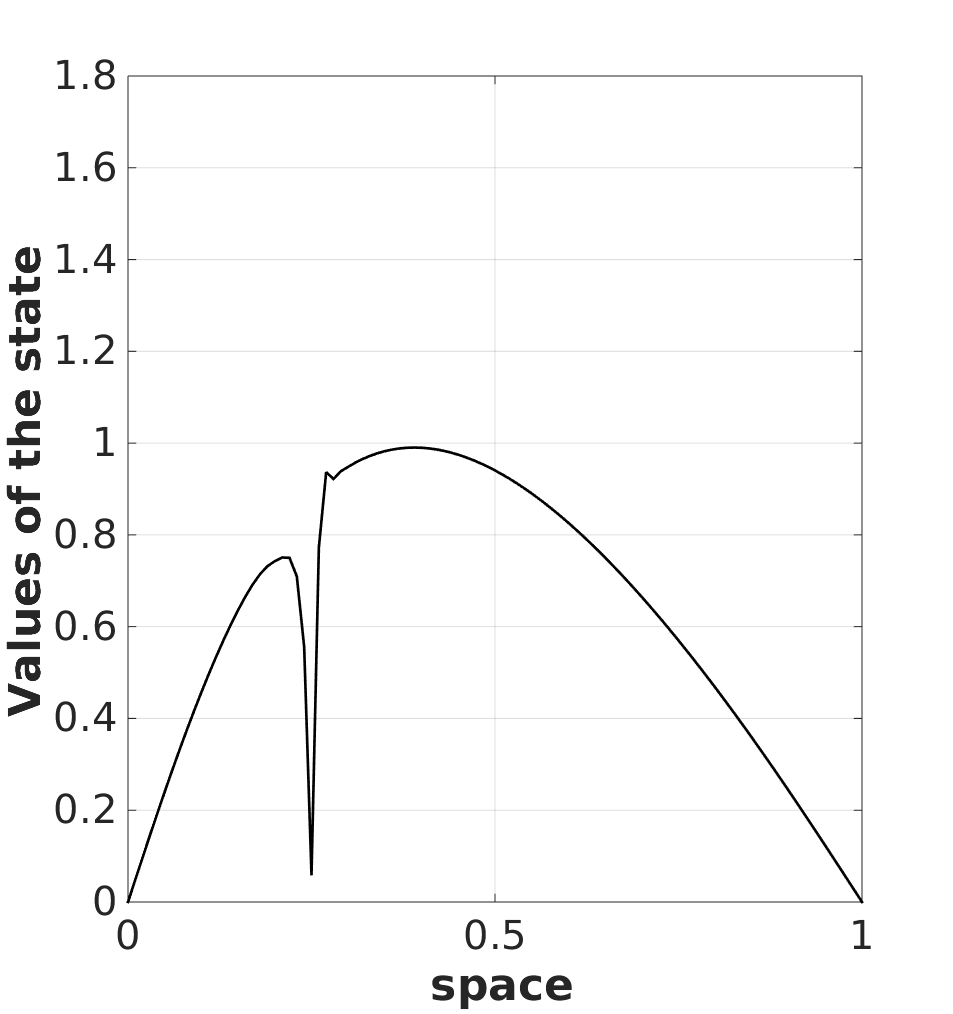} \\
\includegraphics[trim = 0cm 0cm 0cm 1cm, clip, scale=0.15]{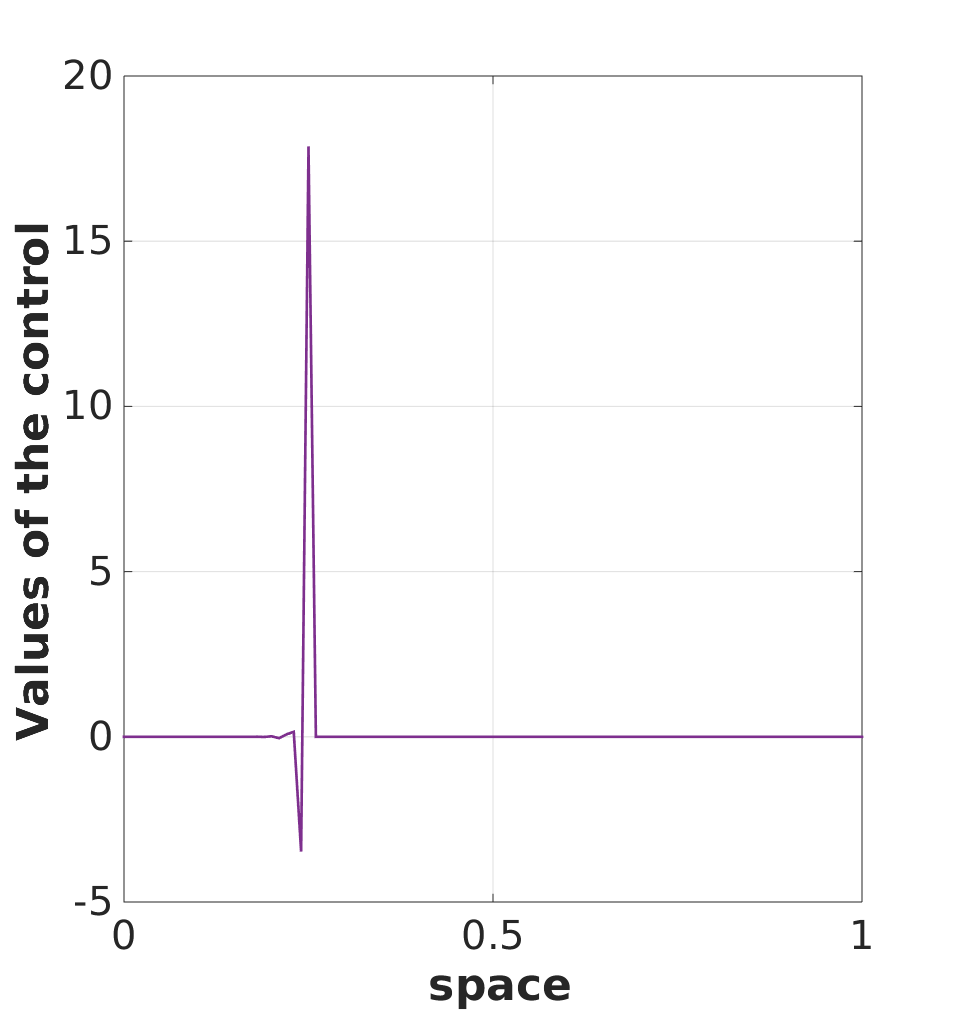}
\begin{center} $ t = 4.7111 $ \end{center}
\end{minipage}
\hspace{0.03\linewidth}
\begin{minipage}{0.2175\linewidth}
\includegraphics[trim = 0cm 0cm 0cm 1cm, clip, scale=0.15]{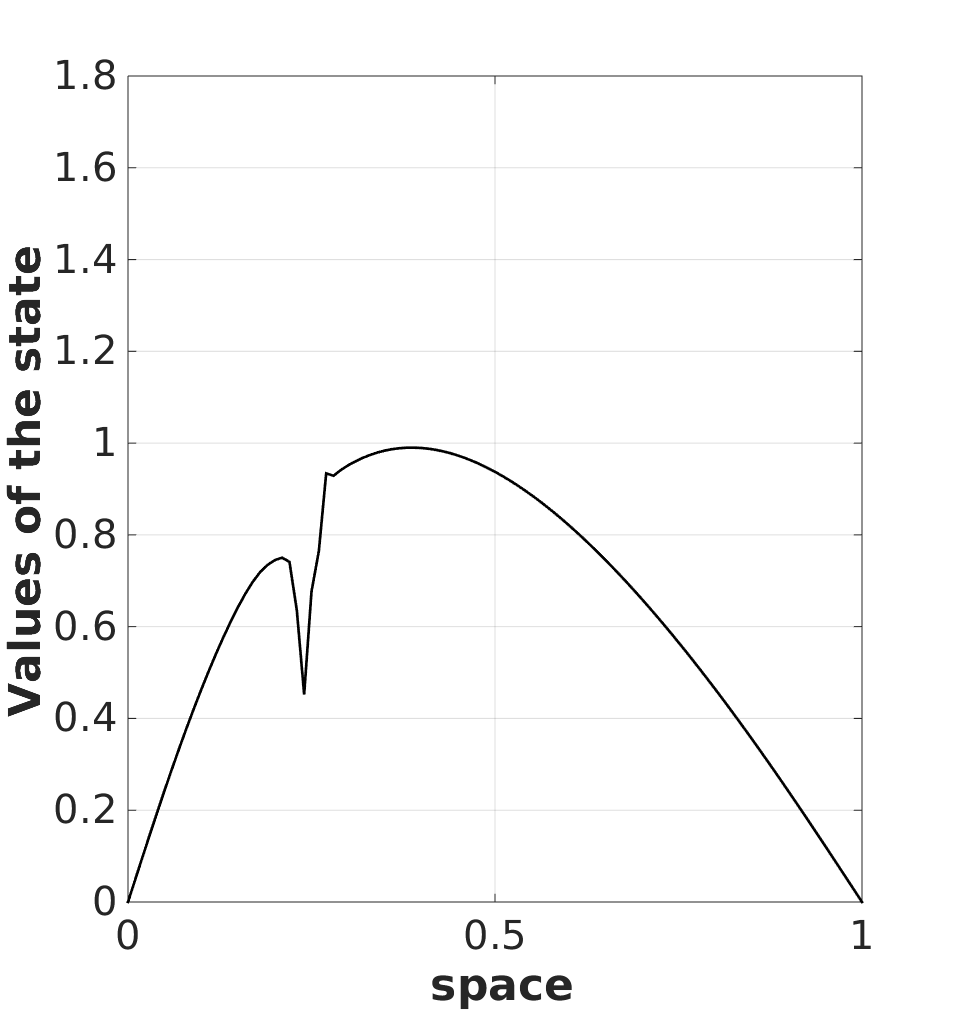} \\
\includegraphics[trim = 0cm 0cm 0cm 1cm, clip, scale=0.15]{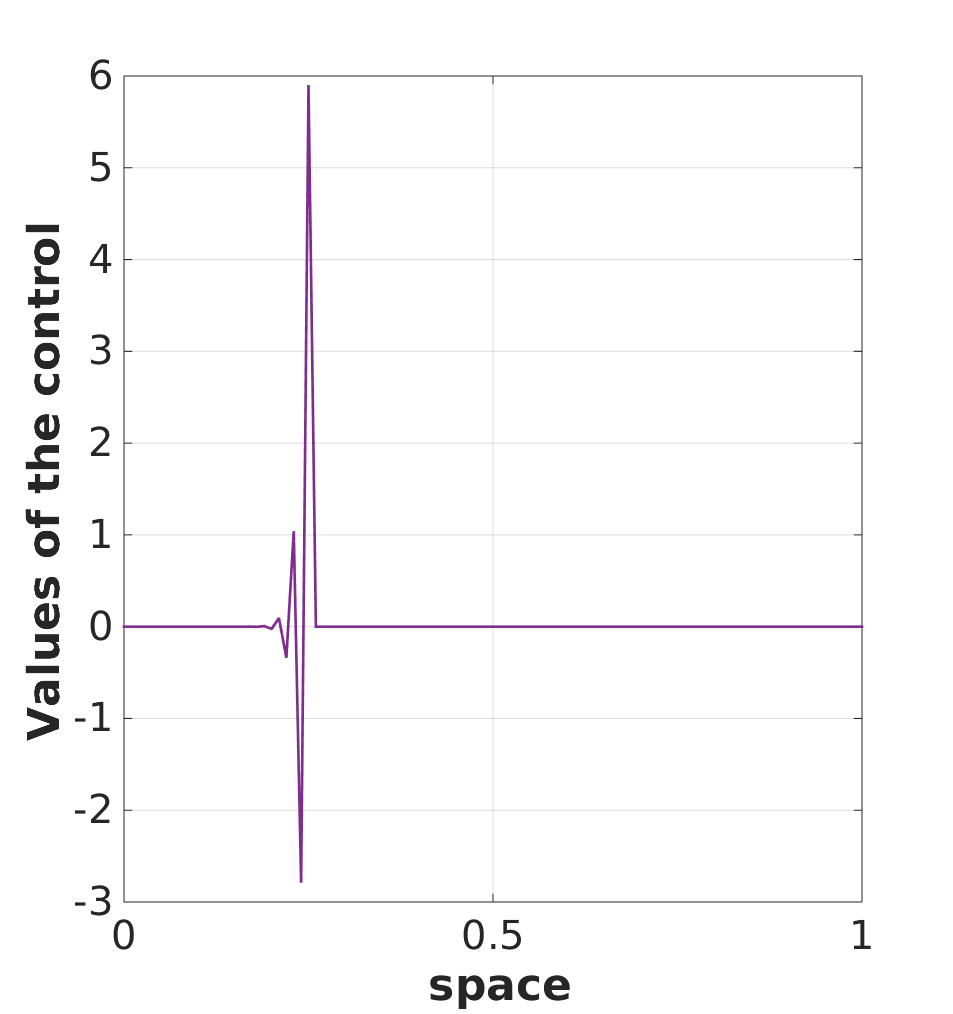}
\begin{center} $ t = 4.7839 $ \end{center}
\end{minipage}
\hspace{0.03\linewidth}
\begin{minipage}{0.2175\linewidth}
\includegraphics[trim = 0cm 0cm 0cm 1cm, clip, scale=0.15]{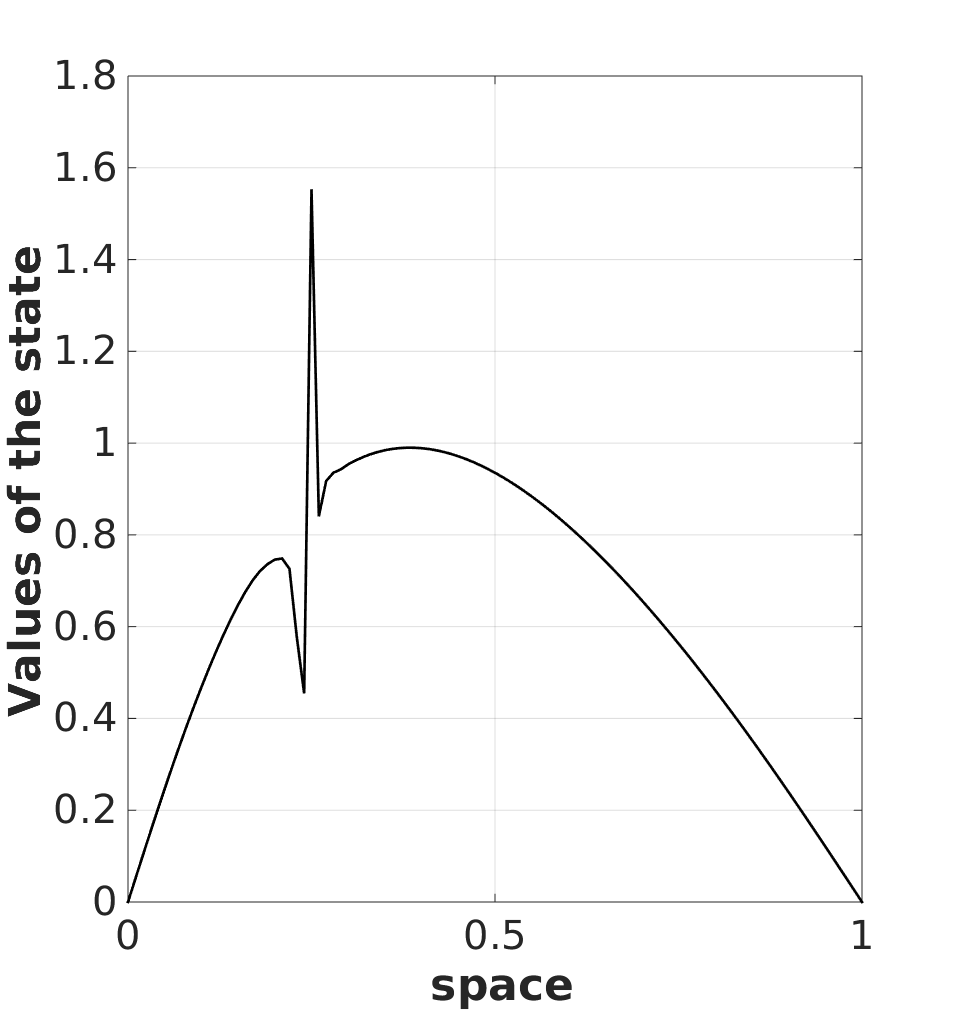} \\
\includegraphics[trim = 0cm 0cm 0cm 1cm, clip, scale=0.15]{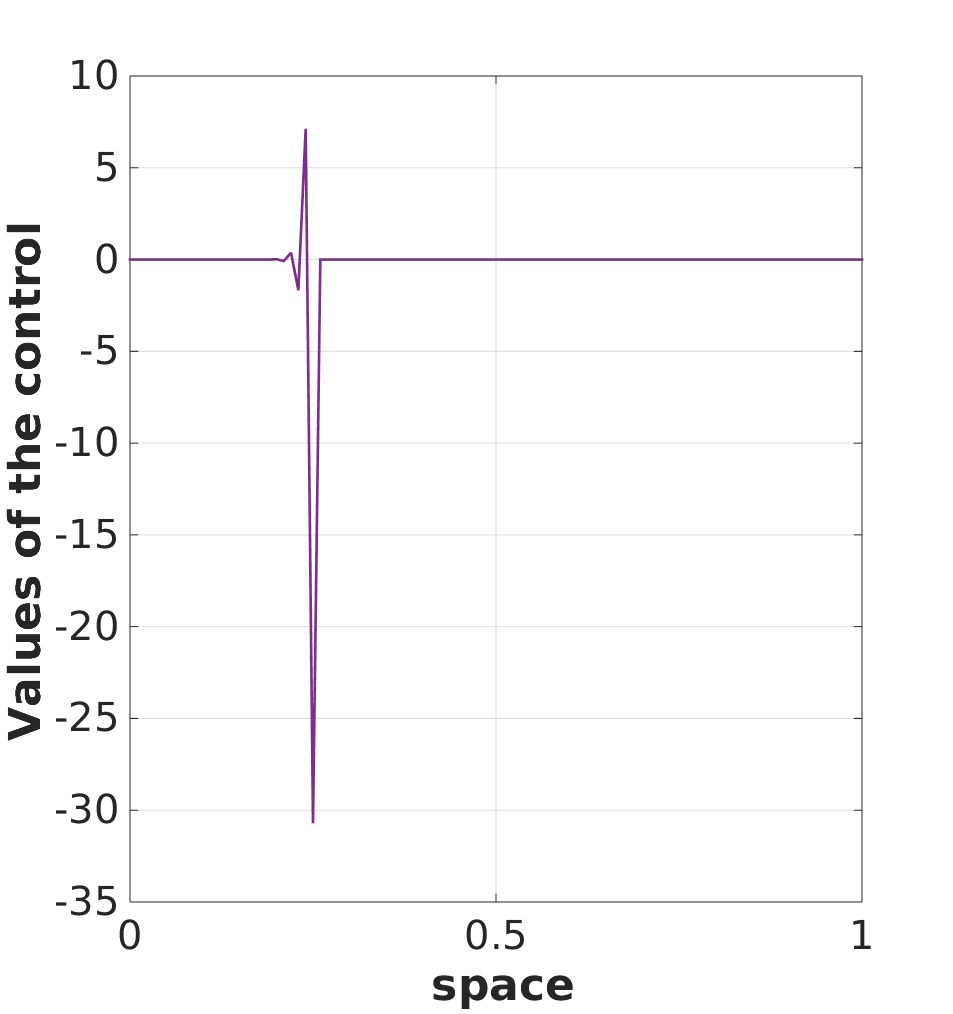}
\begin{center} $ t = 4.8325 $ \end{center}
\end{minipage}  
\\ %\vspace*{-10pt}
       %\end{minipage}
       %\begin{minipage}{\linewidth}
       %\centering
\begin{minipage}{0.2175\linewidth}
\includegraphics[trim = 0cm 0cm 0cm 1cm, clip, scale=0.15]{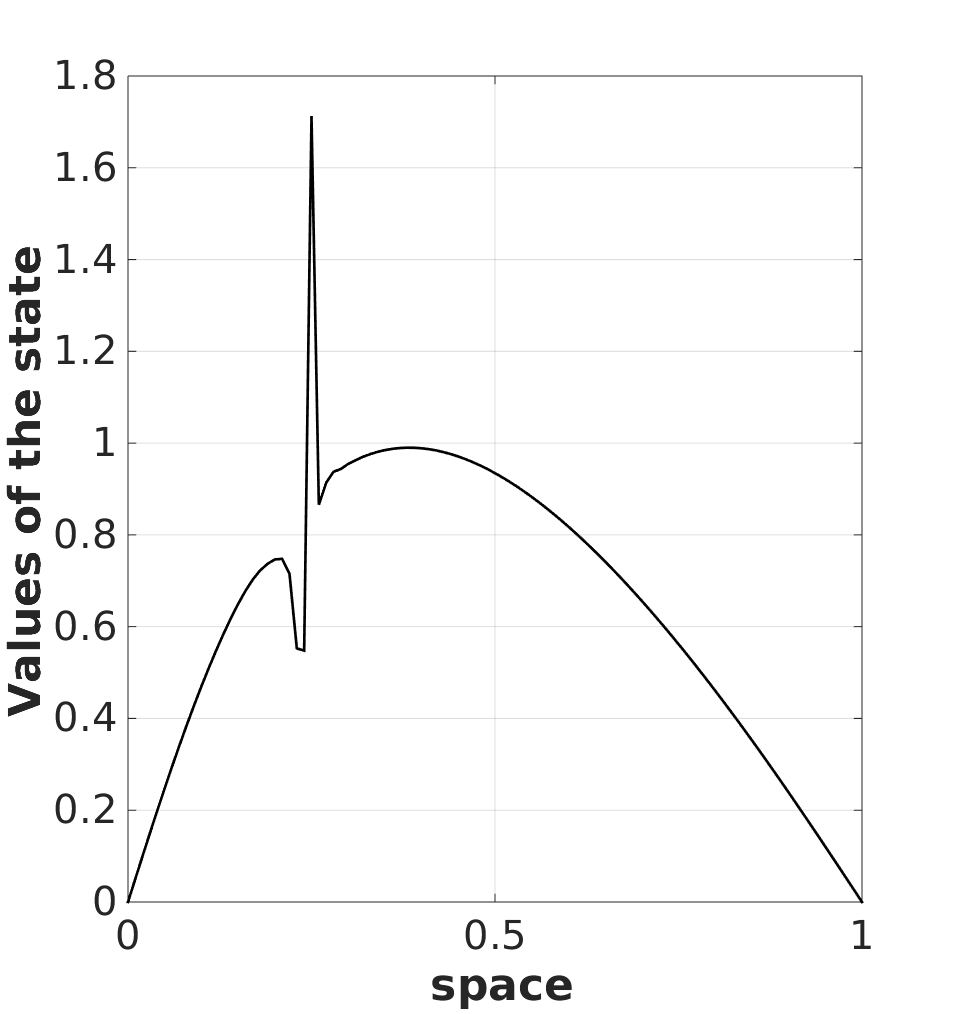} \\
\includegraphics[trim = 0cm 0cm 0cm 1cm, clip, scale=0.15]{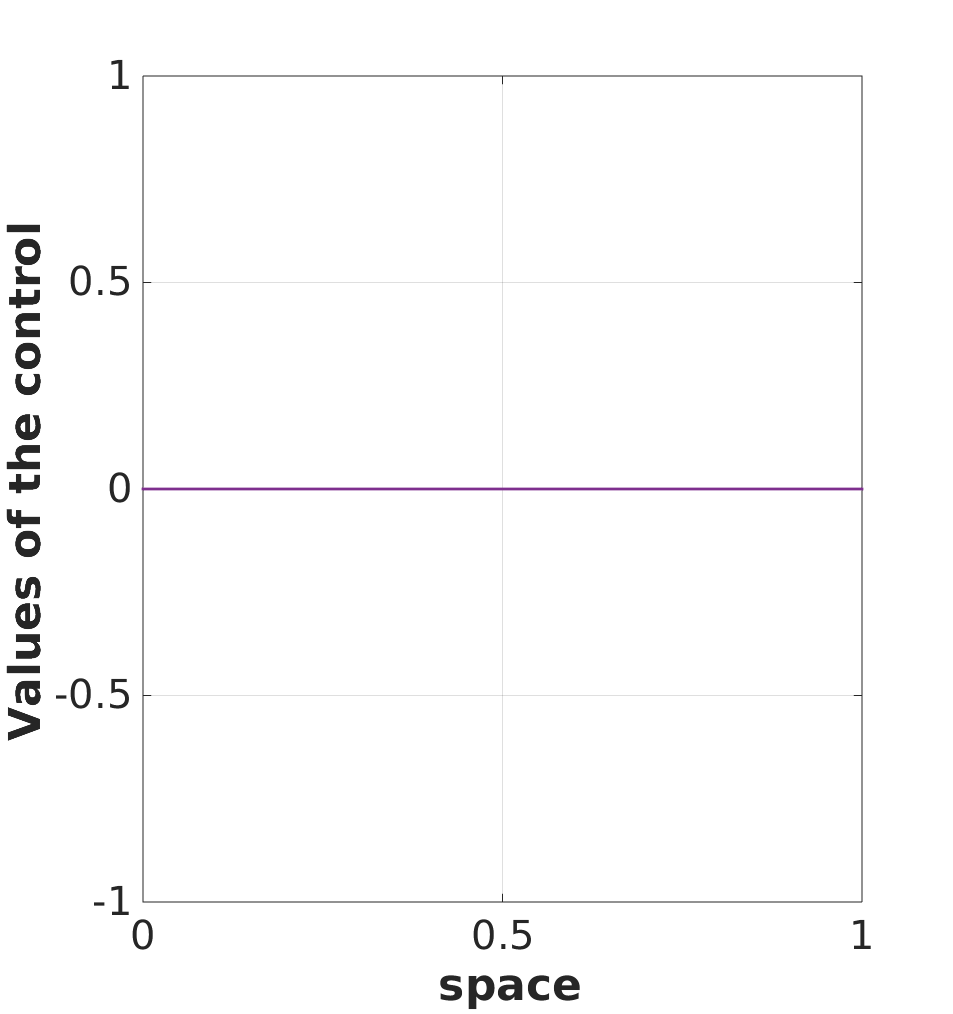}
\begin{center} $ t = 4.8519  $ \end{center}
\end{minipage}
\hspace{0.03\linewidth}
\begin{minipage}{0.2175\linewidth}
\includegraphics[trim = 0cm 0cm 0cm 1cm, clip, scale=0.15]{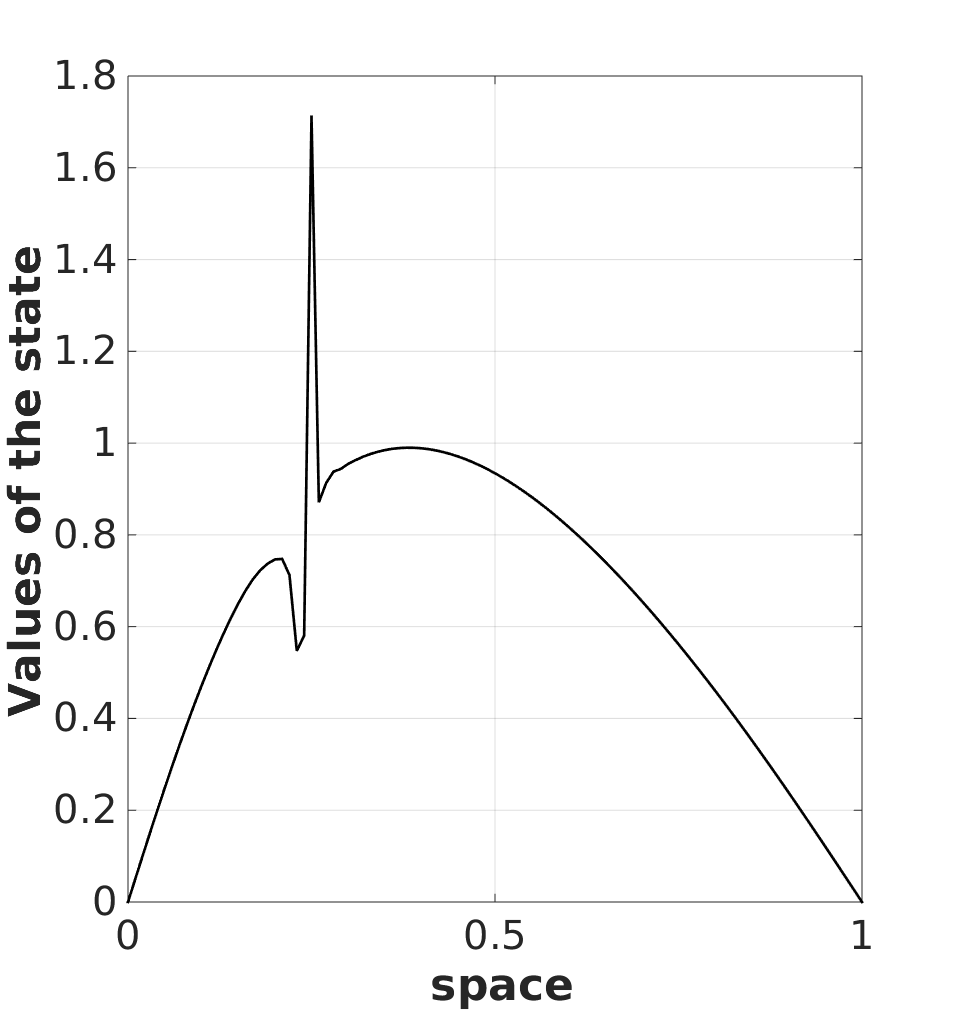} \\
\includegraphics[trim = 0cm 0cm 0cm 1cm, clip, scale=0.15]{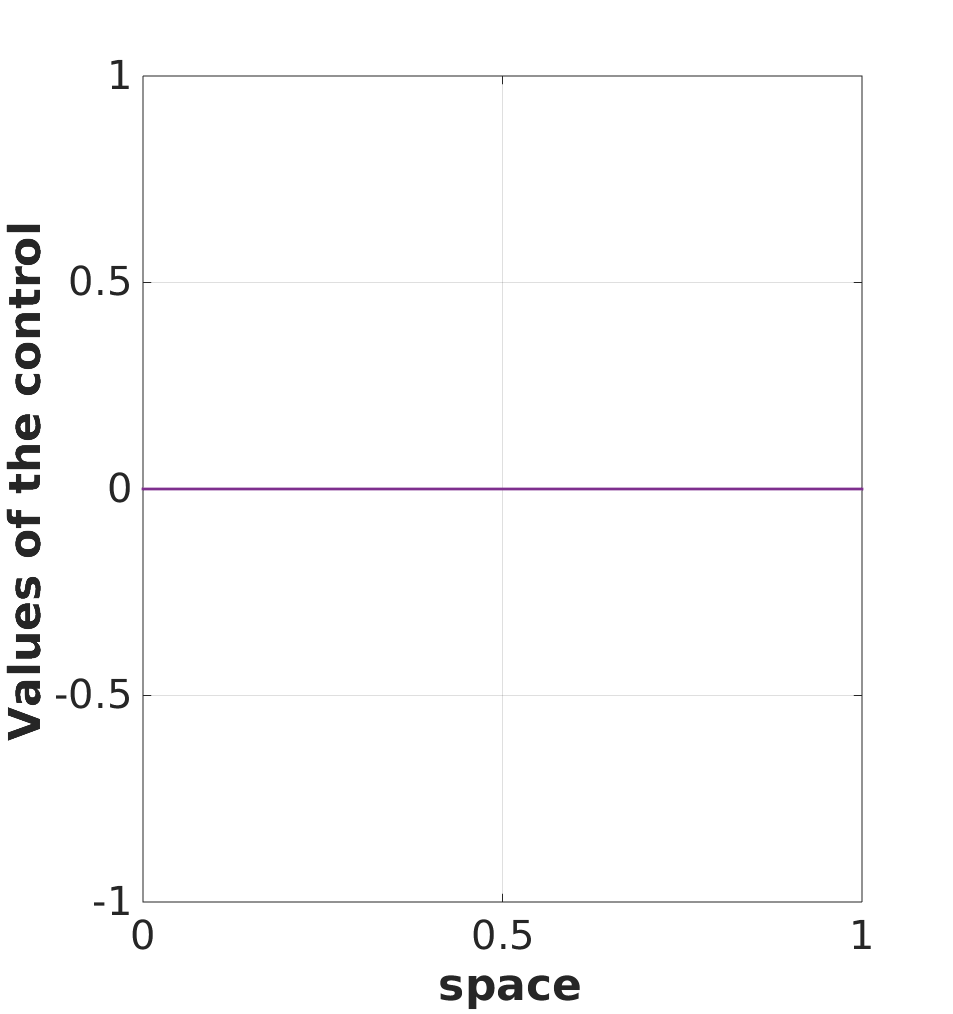}
\begin{center} $ t =  4.8568 \approx \tau $ \end{center}
\end{minipage}
\hspace{0.03\linewidth}
\begin{minipage}{0.2175\linewidth}
\includegraphics[trim = 0cm 0cm 0cm 1cm, clip, scale=0.15]{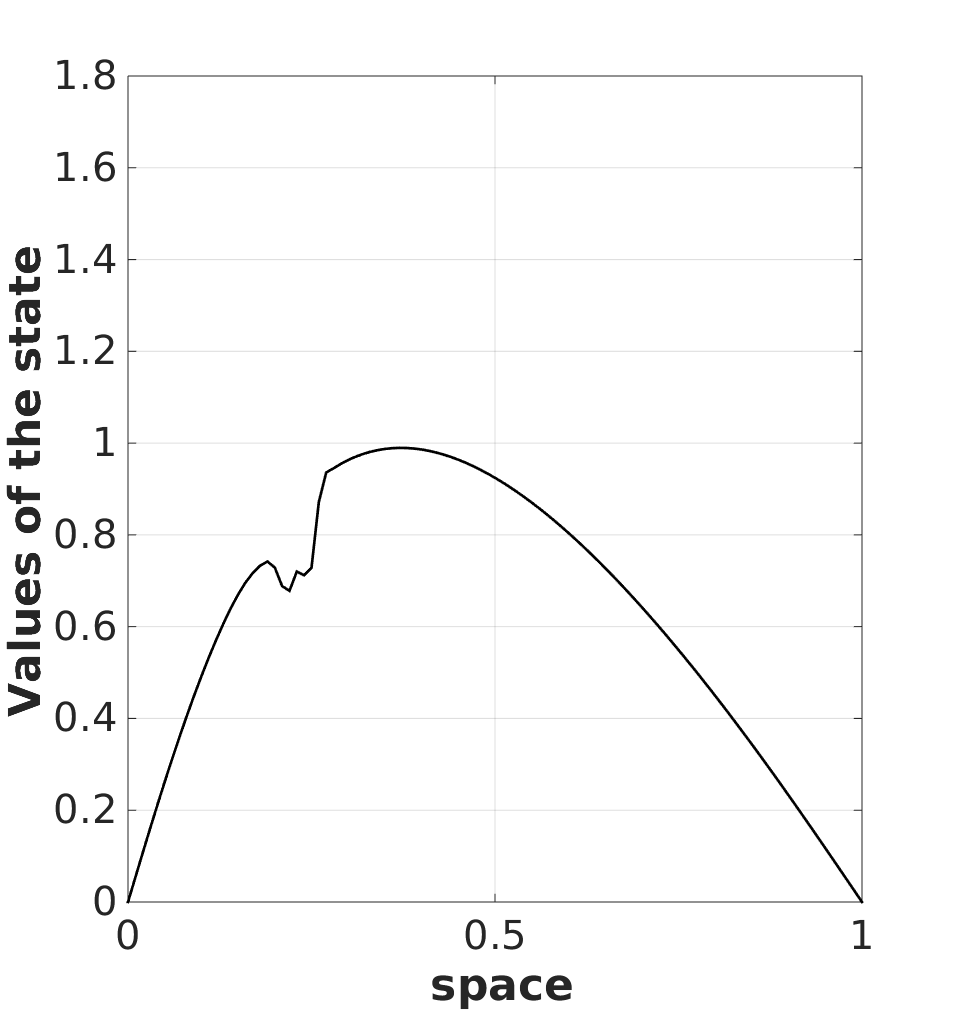} \\
\includegraphics[trim = 0cm 0cm 0cm 1cm, clip, scale=0.15]{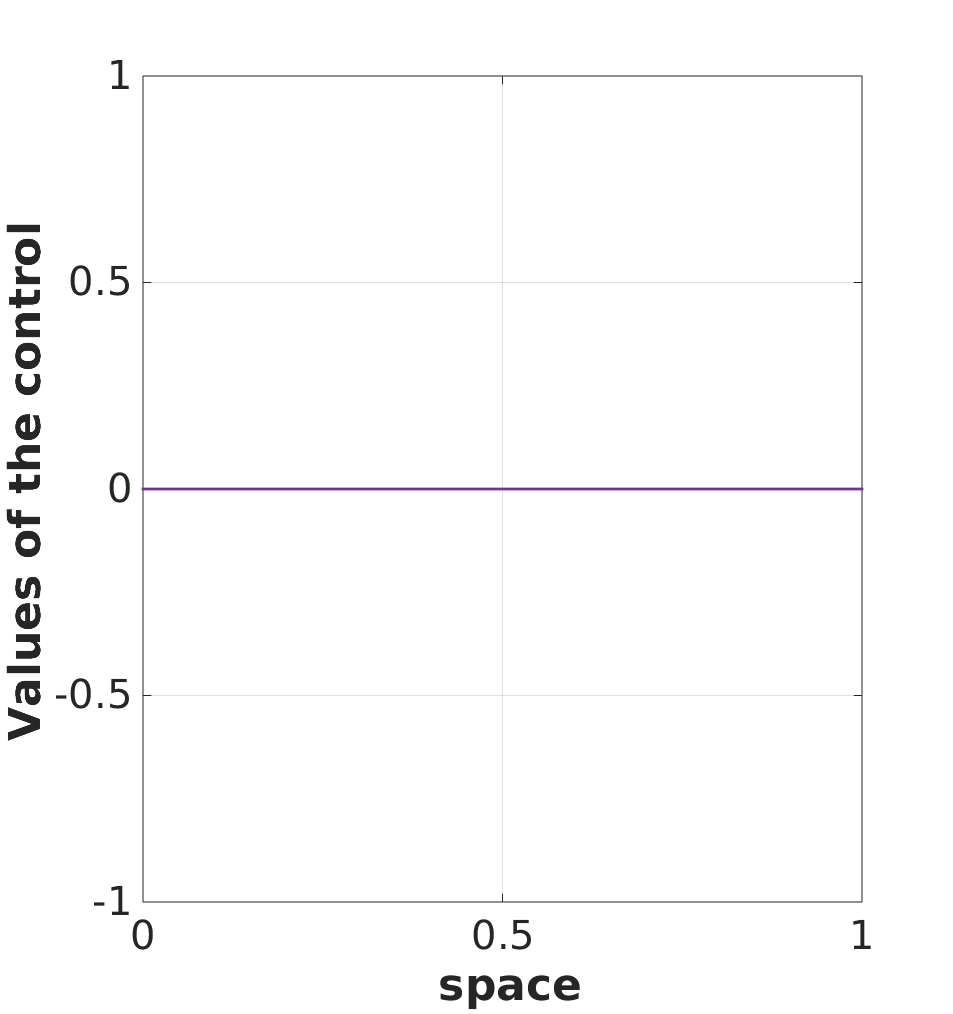}
\begin{center} $ t = 5.0111 $ \end{center}
\end{minipage}
\hspace{0.03\linewidth}
\begin{minipage}{0.2175\linewidth}
\includegraphics[trim = 0cm 0cm 0cm 1cm, clip, scale=0.15]{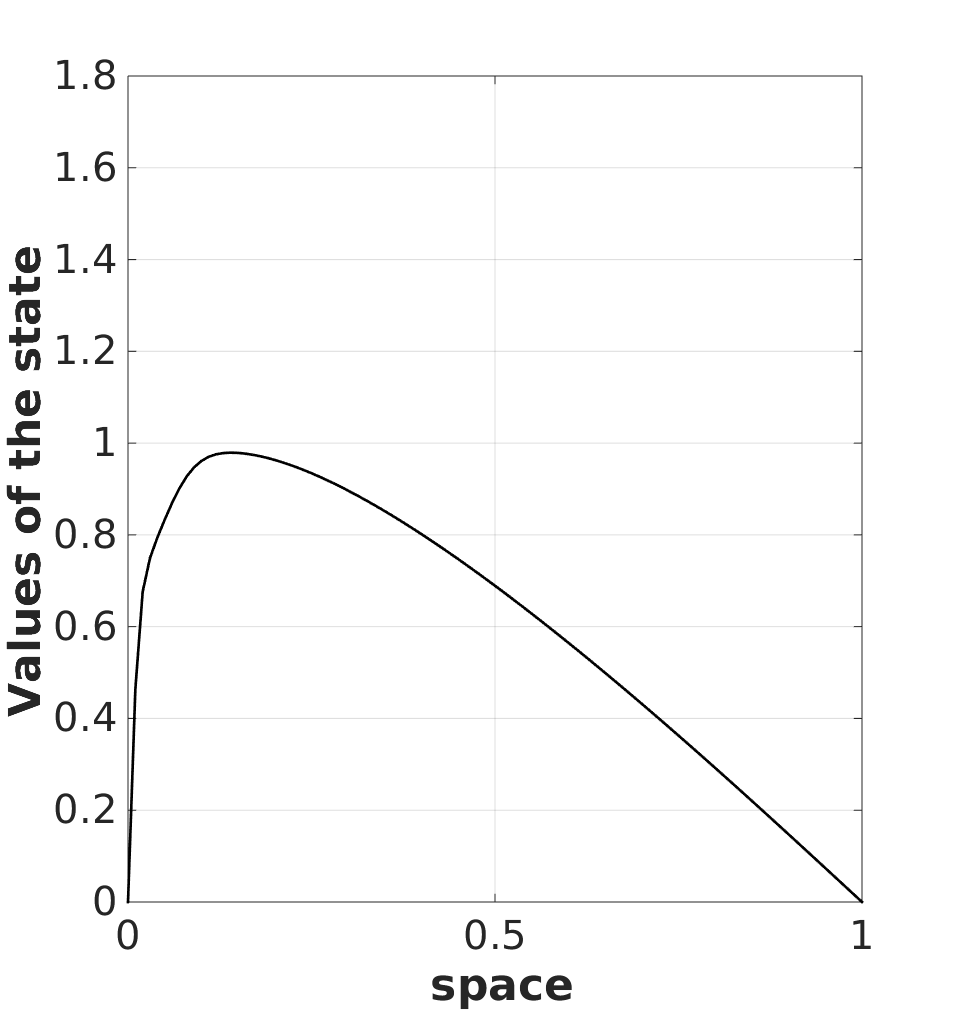} \\
\includegraphics[trim = 0cm 0cm 0cm 1cm, clip, scale=0.15]{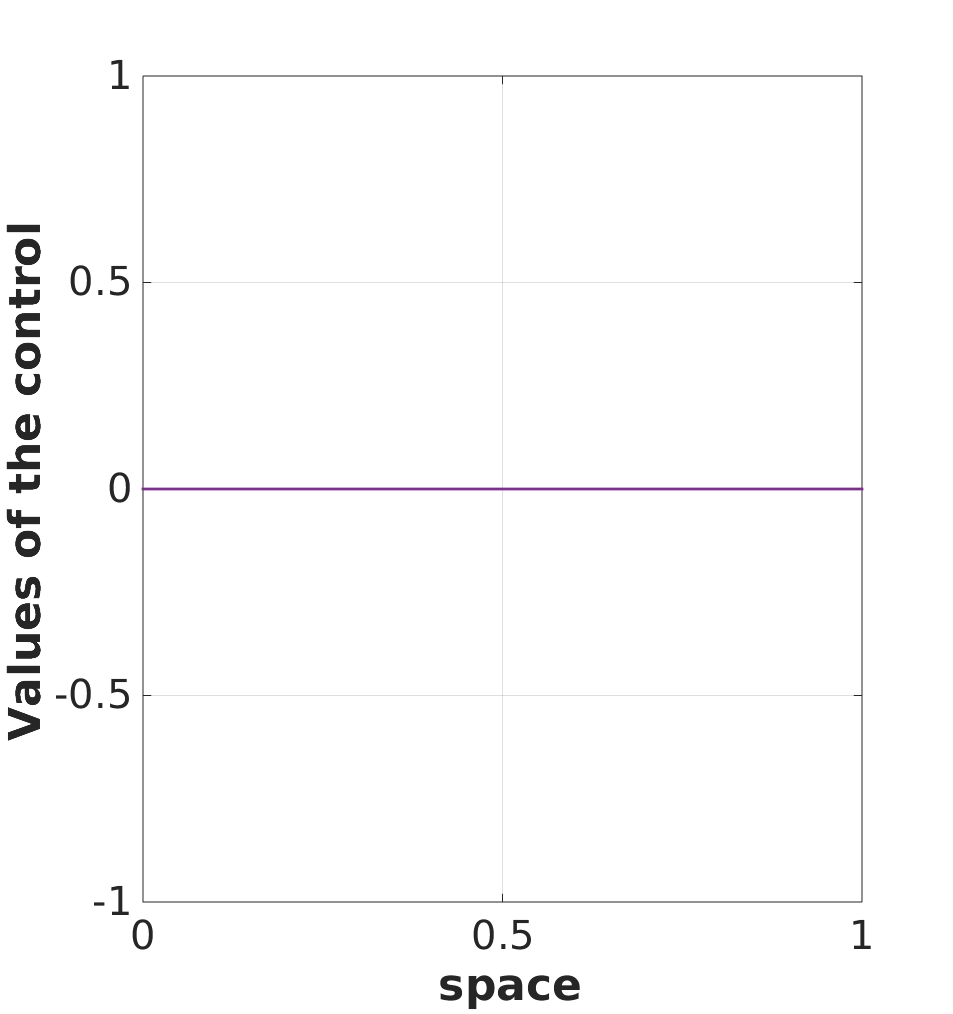}
\begin{center} $ t = 10.0000 $ \end{center}
\end{minipage}  
\end{minipage}
\hfill \\ \hfill \\
\begin{minipage}{\linewidth}
\captionof{figure}{Values of the state and the control, for different values of the time. \textcolor{white}{}\label{figBurgers}}
\end{minipage}\\
\hfill \\
\FloatBarrier

The maximum is reached for $\tau \approx 4.86$. In view of the profile of the initial condition, without control the solution is transported to the left part of the domain. The simulation shows that the control seems to wait for the so transported energy, before being mainly active during a small period before $t = \tau$, operating a bumping effect. Its influence on the state leads to transport some energy from $\omega$ to $D$. For $t > \tau$, the energy obtained on $D$ is diffused into the profile (the viscosity $\nu$ is chosen here very small, so that the diffusion due to the term $\nu y_{xx}$ is almost not noticed). Note that some delay is encoded into the model, that is to say when the maximum is reached for the state, at time $\tau$, and even a little bit before this moment, the control is no longer active.

\section*{Acknowledgments}
The authors gratefully acknowledge support by the Austrian Science Fund (FWF) special
research grant SFB-F32 "Mathematical Optimization and Applications in Biomedical
Sciences", and by the ERC advanced grant 668998 (OCLOC) under the EU's H2020
research program.

%\nocite{*}
\bibliography{Zamm-201600199}%

\end{document}